\documentclass[11pt]{amsart}

\usepackage{amssymb, amsmath, textcomp, amsthm, amsfonts, bm}
\usepackage{bbm,dsfont}
\usepackage{color}
\usepackage{cancel}
\usepackage{hyperref}

\usepackage{lipsum}
\usepackage{comment}

\title[Restriction estimates]{A restriction estimate for surfaces  with negative Gaussian curvatures}
\author{Shaoming Guo and Changkeun Oh}
\date{}

\def\R{\mathbb{R}}

\def\C{\mathbb{C}}\def\nint{\mathop{\diagup\kern-13.0pt\int}}
\def\Z{\mathbb{Z}}
\def\T{\mathbb{T}}
\def\CN{\mathcal{N}}
\def\D{\mathcal{D}}

\def\dt{\delta_{\mathrm{trans}}}
\def\dd{\delta_{\mathrm{deg}}}
\def\avg{L^2_{\mathrm{avg}}(\theta)}

\def\beq{\begin{equation}}
\def\endeq{\end{equation}}
\def\bg{\begin{gathered}}
\def\eg{\end{gathered}}
\def\br{\mathrm{Br}_{\alpha}}

\def\mc{\mathcal}
\def\lesim{\lesssim}

\newcommand{\Norm}[1]{ \left\|  #1 \right\| }

\numberwithin{equation}{section}
\theoremstyle{plain}
\newtheorem{thm}{Theorem}[section]
\newtheorem{prop}[thm]{Proposition}

\newtheorem{lem}[thm]{Lemma}

\newtheorem*{conj*}{Conjecture}

\newtheorem*{openproblem*}{Open Problem}

\begin{document}
\maketitle
\begin{abstract}
     We prove $L^p$ bounds for the Fourier extension operators associated to smooth surfaces in $\R^3$
    with negative Gaussian curvatures for every $p>3.25$.
\end{abstract}

\section{Introduction}
For every hypersurface $S$ in $\R^3$ parametrized by
\beq\label{eq_surface_parametrized}
S=\{(\xi,\eta,h(\xi,\eta): (\xi,\eta) \in [-1,1]^2 \},
\endeq
we define an extension operator associated to the surface $S$ by
\begin{equation}\label{extensionoperator}
E_Sf(x_1,x_2,x_3)=\int_{[-1,1]^2}f(\xi,\eta)
e\big(\xi x_1+\eta x_2+h(\xi,\eta)x_3 \big)\,d\xi d\eta.
\end{equation}
Here, $h(\xi,\eta)$ is a smooth function
and $e(t)=e^{2\pi it}$. 
One formulation of the restriction problem, introduced by Stein in the 1960s, is to find an optimal range of $(p, q)$ satisfying the following Fourier extension estimate
\beq\label{extensionestimate}
\|E_Sf\|_{L^{p}(\R^3)} \leq C(S,p,q) \|f\|_{L^{q}([-1,1]^2)}.
\endeq
It is reasonable to expect that the range of $(p, q)$ depends on the choice of the smooth function $h$.
Stein (see for instance page 345 of \cite{MR2827930}) conjectured
that 
the extension estimate \eqref{extensionestimate} holds true if and only if $p>3$ and $p \ge 2q'$ with $p':=p/(p-1)$, provided that the determinant of the Hessian matrix of $h$ is non-vanishing, that is, 
\beq\label{hessian}
\det \begin{pmatrix}
\frac{\partial^2{h}}{\partial \xi^2} & \frac{\partial^2{h}}{\partial \xi
\partial \eta} 
\\[0.50em]
\frac{\partial^2{h}}{\partial \xi \partial \eta} & \frac{\partial^2{h}}{\partial \eta^2}
\end{pmatrix}\neq 0
\endeq
at every point on $[-1, 1]^2$. 
We refer to \cite{MR2087245,MR3967448} for historical backgrounds and applications of the restriction estimates.\\

The condition \eqref{hessian} is equivalent to the condition that the Gaussian curvature of the surface $S$ 
is non-vanishing everywhere.
As the determinant of the Hessian matrix is continuous, once the condition \eqref{hessian} is satisfied, the sign of the determinant does not change as the point changes in $[-1, 1]^2$. 

For the case that the sign is positive, there has been a number of significant progress over the past twenty years. Tao, Vargas and Vega \cite{MR1625056}  proved \eqref{extensionestimate} for $p\ge 4-5/27$ and the optimal range of $q$, and for $p\ge 4-2/9$ and certain range of $q$. Indeed, they discovered that one can derive the linear restriction estimate by using certain bilinear restriction estimates. 
Tao \cite{MR2033842}  obtained \eqref{extensionestimate} for $p>10/3$ by proving certain sharp bilinear restriction estimates. Bourgain and Guth \cite{MR2860188} invented the method of broad-narrow analysis, which allows them to use the tri-linear restriction estimate due to Bennett, Carbery and Tao \cite{MR2275834}, and improve the range to $p>56/17$. Guth \cite{MR3454378} further developed the idea of broad-narrow analysis, successfully applied the polynomial method in the context of Fourier restriction, and pushed the range to $p>3.25$ (and $q=\infty$). Later, Shayya \cite{MR3694293} and Kim \cite{Kim2017SomeRO} refined the argument of Guth \cite{MR3454378} and improved his result to the range $p>3.25$ and $p\ge 2q'$. The most recent result is due to Wang \cite{Wang2018ARE}, who proved  \eqref{extensionestimate} for $p>3+3/13$. Her method is based on the polynomial method and the two ends argument that originated from the work of Wolff \cite{MR1836285}. For earlier results, we refer to Tao, Vargas and Vega \cite{MR1625056}, in particular, Table 1 on page 969 of their paper. 

Regarding the case of negative Gaussian curvatures, the case of the hyperbolic paraboloid $h(\xi, \eta)=\xi\eta$ has been very well studied. 
Lee \cite{MR2218987} and Vargas \cite{MR2106972} independently obtained \eqref{extensionestimate} for $p>10/3$ by using the bilinear method. Their result has been improved by Cho and Lee \cite{MR3653943} to the range $p>3.25$ and $q=\infty$  by using the polynomial method of Guth \cite{MR3454378}. Later, Kim \cite{Kim2017SomeRO} refined their argument, and improved their result to the range 
$p>3.25$ and $p>2q'$. More recently, Stovall \cite{MR3892401} 
obtained scale-invariant restriction estimates for $p>3.25$ (and $p=2q'$ which makes $(p, q)$ on the scaling line). 

However, for general surfaces with negative Gaussian curvatures, the problem starts to get more complicated. Buschenhenke, M\"{u}ller, and Vargas first studied the problem of restriction estimates for perturbed hyperbolic paraboloids in one variable, given by $h(\xi,\eta)=\xi\eta+g(\eta)$ for some smooth function $g$. A typical and model example is $g(\eta)=\eta^3$. They obtained the restriction estimate \eqref{extensionestimate} 
 for this typical example
in the range $p>10/3$ in \cite{MR3999679} and in the range $p>3.25$ in \cite{arXiv:2003.01619}; for functions $g$ that are of finite types in the range $p>10/3$ in \cite{Buschenhenke2019OnFR}; and for flat functions $g$ with $g'''$ monotone in the range $p>10/3$ in \cite{Buschenhenke2020PartitionsOF}.

In this paper, we obtain a restriction estimate in the range $p>3.25$ for every smooth surface with negative Gaussian curvatures. 
In particular, our result improves over the range $p\ge 4$ of Tomas \cite{MR358216} and Stein \cite{MR864375} for such surfaces, and no prior results seem to be known for two-variable perturbations of the hyperbolic paraboloid. 

\begin{thm}\label{theorem1.1}
Let $h(\xi, \eta)$ be a smooth function such that 
\beq\label{1.4}
\det \begin{pmatrix}
\frac{\partial^2{h}}{\partial \xi^2} & \frac{\partial^2{h}}{\partial \xi
\partial \eta} 
\\[0.50em]
\frac{\partial^2{h}}{\partial \xi \partial \eta} & \frac{\partial^2{h}}{\partial \eta^2} 
\end{pmatrix} < 0,
\endeq
at every point in $[-1, 1]^2$. 
Then for every $p>3.25$, $p>2q'$,
and every function $f:[-1,1]^2 \rightarrow \C$, it holds that 
\beq
\Norm{E_Sf}_{L^{p}(\R^3)}\le C_{h,p,q}\|f\|_{L^q([-1,1]^2)},
\endeq  
where $C_{h, p, q}$ is a constant depending on $h, p$ and $q$. 
\end{thm}

Let us briefly explain why the case of perturbed hyperbolic paraboloid is more difficult than that of the hyperbolic paraboloid. In the latter case, we have two distinguished directions and we have clean scalings that preserve the surface, to be more precise, the surface $(\xi, \eta, \xi\eta)$ is preserved under the scaling $(\xi, \eta)\mapsto (a\xi, b\eta)$ for two arbitrary $a, b\in \R$. This scaling structure plays a crucial role in the previously mentioned works \cite{MR2106972, MR2218987, MR3653943, MR3892401}.  However, in the general case, such a scaling structure is no longer available. For instance, Buschenhenke, M\"uller and Vargas \cite{MR3999679} explained in their paper that under certain natural scalings, the term $\eta^3$ in their surface $(\xi, \eta, \xi\eta+\eta^3)$ could assume a dominant role (see the last equation in page 127 of their paper and the discussion after it) and can no longer be viewed as a perturbation. \\

Next we briefly mention the novelties of the paper in the setting where $h$ is a polynomial. We follow the framework of the polynomial methods of Guth \cite{MR3454378}, which is partly based on the method of induction on scales. The major difference comes from the ``wall" case (see equation \ref{0522e4.9} and the line below it), in particular, how the $L^4$ argument of Cordoba, Sj\"olin and Fefferman (see for instance Lemma 7.1 of Tao \cite{MR2033842}) is applied. When certain ``significant" frequency squares (see \eqref{0522e4.36}) form ``bad lines" or ``bad pairs" (see \eqref{0421e2.3} and Subsection \ref{0522subsection2.2}), we will not have good enough $L^4$ estimate. However, Lemma \ref{linecount} and Lemma \ref{badpairlemma}, which is one main ingredient of the paper, tell us how bad pairs and bad lines are distributed. Moreover, in contrast to the case of the hyperbolic paraboloid, bad lines are no longer either horizontal or vertical, but can point to a large number of directions. This will make the notion of ``broad functions" of Guth \cite{MR3454378} more complicated (see equation \eqref{0511e3.2} and \eqref{0507e3.3}). Moreover, as our polynomial $h(\xi, \eta)$ allows perturbations of all different orders up to the degree $d$, when running the method of induction on scales, we will have to pick as many scales $K_{d+1}\leq K_d \leq \dots \leq K_0$ (see \eqref{0522e2.56}) to take care of all the perturbations. Due to the appearance of these different scales, when defining broad functions, we will have to make sure that our function is broad at every scale. This will make the argument of passing from bilinear estimates to linear estimates more involved (see Lemma \ref{bourgainguth}). \\

{\bf Notation:} For two non-negative numbers $A_1$ and $A_2$, we write $A_1 \lesim A_2$ to mean that there exists a constant $C$, which possibly depends on $d$, such that $A_1 \le C A_2$. Similarly, we use $O(A_1)$ to denote a number whose absolute value is smaller than $CA_1$ for some constant $C$.  

For simplicity, we sometimes use $Eg$ for $E_Sg$ and use $p_0:=3.25$.
For a polynomial $h$, we use $\|h\|$ to denote the supremum of all its coefficients. For every set $A$ and a number $a>0$, we denote by $\mathcal{N}_a(A)$ the $a$-neighborhood of the set $A$. We denote by $1_A$ the characteristic function of $A$ and denote $f1_A$ by $f_A$.  For each square $\tau$ with side length $K^{-1}$, let $2\tau$ denote the square with the same center and the side length $2K^{-1}$. Let $\mathcal{F}$ denote the Fourier transform.
For every $K \geq 1$ and $A \subset [-1,1]^2$ we denote by $\mathcal{P}(K^{-1},A)$ the set of all dyadic squares with side length $K^{-1}$ in $A$. We sometimes use $\mathcal{P}(K^{-1})$ for $\mathcal{P}(K^{-1},[-1,1]^2)$.

 We will use the dyadic numbers $K_L,K_d,\ldots,K_1,K$ with
\begin{equation}
1 \leq K_{L}=:K_{d+1} \leq K_d \leq \cdots  \leq K_1 \leq K_0:= K \leq R.
\end{equation} 
These constants will be determined later. \\

\noindent Note. The authors were reported that Buschenhenke, M\"uller and Vargas \cite{BMV20d} obtained the same result. 

\subsection*{Acknowledgement.} The authors would like to thank Betsy Stovall for a number of insightful discussions. S.G.
was partially supported by NSF grant DMS-1800274.
%

\section{Reduction to polynomial surfaces}

In this section, we will see that in order to prove Theorem \ref{theorem1.1} for smooth functions $h$, it suffices to consider the cases where $h$ is a polynomial. We learned this technique from Section 7.2 of \cite{guth2019}. 

First of all, by Tao's epsilon removal lemma in \cite{MR1666558} (see also \cite{Kim2017SomeRO}), it suffices to prove the local restriction estimate: For every $\epsilon>0$, $p>3.25$ and $p>2q'$, we have 
\begin{equation}\label{200827e2.1}
\|E_S f\|_{L^{p}(B_R)} \le C_{\epsilon, h, p , q} R^{2\epsilon} \|f\|_{L^q([-1, 1]^2)},
\end{equation}
for every ball $B_R\subset \R^3$ of radius $R$. We partition $[-1, 1]^2$ into dyadic squares $\{\Box\}$ of side length $R^{-\epsilon}$. By the triangle inequality, it suffices to prove 
\begin{equation}\label{200827e2.2}
\|E_S f_{\Box}\|_{L^{p}(B_R)} \le C_{\epsilon, h, p , q} \|f_{\Box}\|_{L^q([-1, 1]^2)}.
\end{equation}
Without loss of generality, we assume that $\Box$ contains the origin. Moreover, we assume that $B_R$ is centered at the origin. Denote $W:=10/\epsilon$. Let $h_W$ denote the Taylor expansion of the smooth function $h$ about the origin up to the order $W$. Denote 
\begin{equation}
\Delta_W(\xi, \eta):=h(\xi, \eta)-h_W(\xi, \eta).
\end{equation}
It is easy to see that
\begin{equation}
|\Delta_W(\xi, \eta)|\le \|h\|_{C^{W+1}} R^{-10}, \text{ for every } (\xi, \eta)\in \Box. 
\end{equation} 
Define $S_W$ to be the surface given by $(\xi, \eta, h_W(\xi, \eta))$. By Taylor's expansion, we can write 
\begin{equation}
\begin{split}
E_S f_{\Box}(x) = \int e(x_1\xi+x_2\eta+x_3 h_W(\xi, \eta))\Big(\sum_{k=0}^{\infty} \frac{(ix_3 \Delta_W(\xi, \eta))^k}{k!}\Big)f_{\Box}(\xi, \eta)d\xi d\eta.
\end{split}
\end{equation}
By the triangle inequality, we further obtain 
\begin{equation}
\|E_S f_{\Box}\|_{L^p(B_R)}\le \|E_{S_W} f_{\Box}\|_{L^p(B_R)}+ \sum_{k=1}^{\infty}\frac{1}{k!} \Big( \|h\|_{C^{W+1}} R^{-9}\Big)^k\cdot R^6\|f_{\Box}\|_1. 
\end{equation}
As $\epsilon$ is fixed and $R$ can be chosen to be sufficiently large, depending on $\epsilon$ and $h$, we see immediately that it suffices to prove 
\begin{equation}\label{200827e2.2zz}
\|E_{S_W} f_{\Box}\|_{L^{p}(B_R)} \le C_{\epsilon, h, p , q} \|f_{\Box}\|_{L^q([-1, 1]^2)}.
\end{equation}
This finishes the reduction of the main theorem for smooth functions to that for polynomials. 

\medskip

In the end we make a further reduction by combining a standard scaling argument with the argument in \cite{MR3694293} (see also Theorem 5.3 in \cite{Kim2017SomeRO}). Theorem \ref{theorem1.1} will follow from Theorem \ref{theorem1.2}.
\begin{thm}\label{theorem1.2}
Let $h$  be given by the polynomial
\begin{equation}\label{0519e1.8}
h(\xi,\eta):=\xi\eta+a_{2,0}\xi^2+a_{2,2}\eta^2+
\sum_{i=3}^{d}\sum_{j=0}^{i} a_{i,j}\xi^{i-j}\eta^{j}.
\end{equation}
For sufficiently small $\epsilon_0>0$, the following holds true: Suppose that the polynomial $h$ satisfies the condition
\beq\label{maininequality}
|a_{2,0}|+|a_{2,2}|+100^d\sum_{i=3}^{d}\sum_{j=0}^{i}|a_{i,j}|\le \epsilon_0.
\endeq
Then for every $3/13<\lambda <1$, $\epsilon>0$, $R \geq 1$, ball $B_R$ of radius $R$, and function $f:[-1,1]^2 \rightarrow \C$, it holds that 
\beq\label{eq_desired_restriction_estimate}
\Norm{E_Sf}_{L^{3.25}(B_R)}\le C_{\epsilon,d,\lambda}R^{\epsilon}\|f\|_{L^2([-1,1]^2)}^{1-\lambda}
\|f\|_{L^{\infty}([-1,1]^2)}^{\lambda}.
\endeq
Here, the constant $C_{\epsilon,d,\lambda}$ is independent of the choice of $a_{i,j}$.
\end{thm}

\subsubsection*{Proof of Theorem \ref{theorem1.1} for polynomial $h$ by assuming Theorem \ref{theorem1.2}}
By a standard scaling argument, we may assume that our polynomial $h$ satisfies the assumption \eqref{maininequality}. 
By the epsilon removal lemma, it suffices to prove the local restriction estimate: For every $p>3.25$, $p>2q'$, and $\epsilon>0$,
\begin{equation}
\|E_Sf\|_{L^p(B_R)} \leq C_{p,q,\epsilon}R^{\epsilon}\|f\|_{L^q([-1,1]^2)}.
\end{equation}
Theorem \ref{theorem1.2} yields the restricted strong type $(p,q)$ for the extension operator $E_S$ at every $p>3.25$ and $p>2q'$. It suffices to interpolate this with the trivial $L^1 \rightarrow L^{\infty}$ estimate.
\qed
\\

\section{Bad lines and bad pairs}

In this section, we introduce the notions {of} bad lines and bad pairs (of frequency points and frequency squares). These concepts are {introduced to study certain degenerate behaviours that originate from higher order perturbations of the hyperbolic paraboloid (see the higher order terms in \eqref{0519e1.8}).} Throughout the rest of the paper, we let $S$ be a hypersurface given by \eqref{eq_surface_parametrized} with a polynomial $h$ satisfying \eqref{maininequality}.

\subsection{Bad lines} 
 Let $c_1(K_L^{-1})$ be a small number given by 
\begin{equation}
    c_1(K_L^{-1}):=10^{-10d}\epsilon_{0}K_L^{-1}.
\end{equation}
We take a collection $\mathbb{D}$ of points on the unit circle $S^1$ such that $\mathbb{D}$ is a maximal $c_1(K_L^{-1})$-separated set. For every $a\in c_1(K_L^{-1})\Z \cap [-10,10]$ and $v=(v_1,v_2) \in \mathbb{D}$ with $|v_2| \leq |v_1|$, let $l_{1,a,v}$ denote the line passing through the point $(0,a)$ with direction vector $v$. Similarly, for every $a \in c_1(K_L^{-1})\Z \cap [-10,10]$ and $v=(v_1,v_2) \in \mathbb{D}$ with $|v_2|>|v_1|$, let $l_{2,a,v}$ denote the line passing through $(a,0)$ with direction vector $v$. We distinguish these two collections of lines as we will see that the variables $\xi$ and $\eta$ will be rescalled differently. This is a phenomenon that was already present in Lee \cite{MR2218987} and Vargas \cite{MR2106972}.\\

Let us now define bad lines of the hypersurface $S$.
Suppose that a line $l=l_{1,a,v}$ is given. We define an affine transformation $M_l$ associated to the line $l$ in the following way. First, let $M_{l,1}$ be the action of translation that sends $(0,a)$ to the origin and let $M_{l,2}$ be the rotation mapping $v$ to the point $(1,0)$. Define
\begin{equation}
    M_l:=M_{l,2} \circ M_{l,1}.
\end{equation}
We write the polynomial $(h \circ M_l^{-1})(\xi,\eta)$ as 
\begin{equation}\label{0421e2.2}
(h \circ M_l^{-1})(\xi,\eta)=\sum_{i=0}^d\sum_{j=0}^{i}c_{i,j}\xi^{i-j}\eta^j.
\end{equation}
The line $l_{1,a,v}$ is called \textit{a bad line} provided that
\begin{equation}\label{0421e2.3}
\max_{i=2,\ldots,d}(|c_{i,0}|) \leq 10^{-5d}\epsilon_0K_L^{-1}=10^{5d} c_1(K_L^{-1}).
\end{equation}
We define a bad line for $l_{2,a,v}$ in a similar way, with the role of $\xi$ and $\eta$ in \eqref{0421e2.2} and \eqref{0421e2.3} exchanged.\\

Take $\iota\in \{1, 2\}$. Let $L_{\iota,a,v}$ denote the $c_1(K_L^{-1})$-neighborhood of $l_{\iota,a,v}$ and call it \textit{a bad strip}.
We denote {the} collection of all the bad strips by
\begin{equation}
\mathbb{L}:=\{ 
L_{\iota,a,v }: l_{\iota,a,v} \text{ is a bad line}; \iota=1, 2
\}.
\end{equation}
{The directions of bad lines change as the polynomial $h$ changes.} However, the total number of bad lines can always be controlled. Also, bad lines spread out.

\begin{lem}\label{linecount}
Let $K_L$ be a large number.
The total number of bad lines is at most $10^{10d}(c_1(K_L^{-1}))^{-1}$.
For every square $\tau_{L}$ of side length $K_{L}^{-1}$, the number of bad strips intersecting $\tau_L$ is at most $10^{20d}\epsilon_0^{-1}$.
\end{lem}

\begin{proof}[Proof of Lemma \ref{linecount}]
The first statement follows from
\begin{gather}\label{pointbadestimate}
\#\{l_{\iota,a,v}: l_{\iota,a,v} \text{ is a bad line}  \} \leq 10^{8d}
\end{gather}
for every $\iota\in \{1, 2\}$ and every $a \in c_1(K_L^{-1})\Z$. Without loss of generality, we take $\iota=1$.  We write the line $l:=l_{1,a,v}$ as 
\begin{equation}
    \{(\xi, \eta): \eta=(\tan{\theta}) \xi+a\} 
\end{equation}
for some $\theta$ with $|\theta| \leq \pi/4$.
We take the translation $M_{l,1}^{-1}:(\xi,\eta) \mapsto (\xi,a+\eta)$ and the rotation
\begin{equation}
M_{l,2}^{-1}:(\xi,\eta) \mapsto (\xi \cos\theta-\eta\sin\theta, \xi\sin\theta+\eta\cos\theta)
\end{equation}
so that $M_l^{-1}=M_{l,1}^{-1}
\circ
M_{l,2}^{-1}$. 

Let us compute the coefficient of $\xi^2$ of the polynomial $(h \circ M_l^{-1})(\xi,\eta)$. 
The quadratic part of $(h \circ M_{l,1}^{-1})(\xi,\eta)$ is given by 
\begin{equation}
\begin{split}
\Big(
a_{2,0}+
 \sum_{i=3}^{d}a_{i,i-2}a^{i-2}\Big)\xi^2
 &+
 \Big(1+\sum_{i=3}^{d}(i-1)a_{i,i-1}a^{i-2}\Big)\xi\eta
 \\&+\Big(a_{2,2}+\sum_{i=3}^{d}
 \frac{i(i-1)}{2} a_{i,i}a^{i-2}\Big)
 \eta^2.
\end{split}
\end{equation}
The coefficient of $\xi^2$ of the polynomial 
$
(h \circ  M_{l,1}^{-1} \circ M_{l,2}^{-1} )(\xi,\eta) 
$
is given by
\begin{equation}\label{2.10}
\begin{split}
    \Big(
a_{2,0}+
 \sum_{i=3}^{d}a_{i,i-2}a^{i-2}\Big) (\cos \theta)^2
 &+\Big(1+\sum_{i=3}^{d}(i-1)a_{i,i-1}a^{i-2}\Big)\sin \theta \cos \theta
 \\&+\Big(a_{2,2}+\sum_{i=3}^{d}
 \frac{i(i-1)}{2} a_{i,i}a^{i-2})
 \Big)(\sin \theta)^2.
\end{split}
\end{equation}
We denote this function by $F(\theta)$.
By focusing only on the quadratic part, \eqref{pointbadestimate} follows from
\begin{equation}\label{2.11}
\#\{\theta \in c_1(K_L^{-1})\Z \cap [-\pi/4,\pi/4]: |F(\theta)| \leq
10^{-3d}\epsilon_0K_L^{-1}  \} \leq 10^{8d}.
\end{equation}
If $\epsilon_0^{1/2} \leq |\theta| \leq \pi/4$, by the assumption \eqref{maininequality}, we obtain $|F(\theta)| \gtrsim 1$ and the line $l_{1,a,v}$ is not a bad line. Hence, we may assume that $|\theta| \leq \epsilon_0^{1/2}$. 
Note that
\begin{equation}\label{2.12}
\begin{split}
F'(\theta)&=\Big(1+\sum_{i=3}^{d}(i-1)a_{i,i-1}a^{i-2}\Big)\cos 2\theta
 \\&+
\Big(-
(a_{2,0}+
 \sum_{i=3}^{d}a_{i,i-2}a^{i-2})
 +(a_{2,2}+\sum_{i=3}^{d}
 \frac{i(i-1)}{2} a_{i,i}a^{i-2})
 \Big)\sin 2\theta.
\end{split}
\end{equation}
The assumption \eqref{maininequality} implies that $F'(\theta) \geq 1/2$ for every $|\theta|\le \epsilon_0^{1/2}$. The inequality \eqref{2.11} follows immediately. 
\medskip

Let us prove the second statement. Suppose that a square $\tau_{L}$ is fixed. Denote by $(c_1,c_2)$ the left bottom corner of $\tau_L$. By \eqref{pointbadestimate}, it suffices to show that there are at most $10^{11d}\epsilon_0^{-1}$ many $a$ such that $l_{\iota,a,v}$ intersects $\tau_L$ for some $v$. Without loss of generality, we may assume that $\iota=1$. We can write $l_{1,a,v}$ intersecting $\tau_L$ as
\begin{equation}
\{
    (\xi, \eta): \eta=(\tan{\theta}) \xi+(c_2-c_1\tan \theta)+(e_2-e_1\tan \theta)\}
\end{equation}
for some $e_1$ and $e_2$ with $0 \leq e_1,e_2 \leq K_L^{-1}$. It suffices to note that 
\begin{equation}
|e_2-e_1 \tan \theta| \leq 2K_L^{-1}
\end{equation}
and the points $a$ are  $ 10^{-10d}\epsilon_0 K_L^{-1}$-separated.
\end{proof}

We fix $\epsilon_0,\epsilon>0$ and aim to prove Theorem \ref{theorem1.2}. The strategy of the proof is to apply the induction on scales. 
Our induction hypothesis is as follows: For every $3/13< \lambda <1$, surface $S$ with a polynomial $h$ satisfying \eqref{maininequality},  $R'\le R/2$, ball $B_{R'}$ of radius $R'$, and function $f:[-1,1]^2 \rightarrow \C$, we assume 
\begin{equation}\label{induction}
    \|E_S f\|_{L^{3.25}(B_{R'})}
    \le C_{\epsilon,d,\lambda} (R')^{\epsilon}\|f\|_{L^2([-1,1]^2)}^{1-\lambda} \|f\|_{L^{\infty}([-1.1]^2)}^{\lambda}.
\end{equation}
Under this induction hypothesis, we will prove \eqref{eq_desired_restriction_estimate}. 
\\

We will use two different types of rescaling argument: Isotropic rescaling and anisotropic rescaling. 
When a function is supported on a small square, we can make use of the induction hypothesis more effectively via isotropic rescaling. When a function is supported on a bad strip, anisotropic rescaling will come into play. In particular, such anisotropic rescaling is also compatible with our induction hypothesis. 

For every set $A \subset [-1,1]^2$ we denote $g 1_{A}$ by $g_{A}$. Here, $1_A$ is the characteristic function of $A$.
\begin{lem}\label{squarerescaling}
Suppose that $10^{10d} \leq K \leq R' \leq R/2$.
 Under the induction hypothesis \eqref{induction}, for every $3/13<\lambda<1$, surface $S$ with a polynomial $h$ satisfying \eqref{maininequality},  ball $B_{R'}$ of radius $R'$,  square $\tau \subset [-1,1]^2$ with side length $K^{-1}$, and  function $f:[-1,1]^2 \rightarrow \C$, we obtain
\begin{equation}
\|E_Sf_{\tau}\|_{L^{3.25}(B_{R'})} \leq
K^{-1-\lambda+\frac{4}{3.25}}CC_{\epsilon,d,\lambda}(R')^{\epsilon}
\|f_{\tau}\|_{L^2}^{1-\lambda}
\|f_{\tau}\|_{L^\infty}^{\lambda},
\end{equation}
where $C$ is some universal constant.
\end{lem}

\begin{proof}[Proof of Lemma \ref{squarerescaling}] The proof of this lemma is fairly standard. For example, we refer to the proof of Lemma 2.5 of \cite{MR3454378}.
Without loss of generality, we assume that the center of the ball $B_{R'}$ is the origin.
Write $\tau=[c_1,c_1+K^{-1}] \times [c_2,c_2+K^{-1}]$.
 We apply the change of variables: 
 \begin{equation}
 (\xi,\eta) \mapsto \big(K(\xi-c_1),K(\eta-c_2)\big)
 \end{equation}
 and obtain
\begin{equation}
 \begin{split}
 |E_Sf_{\tau}(x)|
 =
 \Bigg|\frac{1}{K^2}\int_{[-1,{1}]^2} g(\xi, \eta)
 e\Big(\big\langle Mx, (\xi,\eta,{h}_1(\xi,\eta)) \big\rangle \Big)
  d\xi d\eta \Bigg|,
 \end{split}
\end{equation}
 where 
 \begin{equation}
     g(\xi, \eta):={f}_{\tau}(c_1+\xi/K,c_2+\eta/K),
 \end{equation}
 $M$ is a linear transformation of the form
\begin{equation}
 Mx=\Big(\frac{x_1+O(1)x_3}{K},\frac{x_2+O(1)x_3}{K},\frac{x_3}{K^2}\Big)
\end{equation}
 and
 ${h}_1$ is the polynomial {given} by
\begin{equation}
 \begin{split}
     {h}_1(\xi,\eta):=
     \big(1+A_1 \big)\xi \eta
     +(A_{2,0})\xi^2+(A_{2,2})\eta^2
     +
     \sum_{i=3}^{d}
     \frac{1}{K^{i-2}}
     \sum_{j=0}^{i}
     (A_{i,j}) \xi^{i-j}\eta^j,
 \end{split}
\end{equation}
 with 
\begin{equation}\label{322}
 \begin{split}
 A_1 &:= \sum_{i=3}^{d}\sum_{j=1}^{i}{(i-j)j}c_1^{i-j-1}c_2^{j-1}a_{i,j},
 \\
A_{i,j}&:= 
 \sum_{i'=i}^{d}\sum_{j'=j}^{i'} \binom{j'}{j}\binom{i'-j'}{i-j}c_1^{i'-j'-i+j}c_2^{j'-j}a_{i',j'}.
 \end{split}
\end{equation}
Note that $|A_1| \lesssim \epsilon_0$.
We define the new surface 
\begin{equation}
    S_1:=\{(\xi, \eta, (1+ A_1)^{-1}h_1(\xi, \eta)): (\xi, \eta)\in [-1, 1]^2\}.
\end{equation}
Via a simple change of variables, we obtain 
\begin{equation}
\begin{split}
\|E_Sf_{\tau}\|_{L^{3.25}(B_{R'})} &\lesssim
K^{-2+\frac{4}{3.25}}
\|E_{S_1}{g}\|_{L^{3.25}(B_{R'})}.
\end{split}
\end{equation}
The ball $B_{R'}$ on the right hand side could have been replaced by a smaller ball. However, we will not take advantage of that gain. Let us check that the polynomial $(1+A_1)^{-1}{h}_1$ satisfies the induction hypothesis \eqref{maininequality}. We need to check that
 \begin{equation}
 \begin{split}
     |A_{2,0}|+|A_{2,2}| +100^d\Big(\sum_{i=3}^d \frac{1}{K^{i-2}}\sum_{j=0}^{i}|A_{i,j}| \Big) \leq 
       \big(1+A_1 \big) \epsilon_0.
 \end{split}
\end{equation}
 This follows from
 \begin{equation}\label{326}
 \begin{split}
     |A_{2,0}|+|A_{2,2}|
     +\epsilon_0|A_1|+100^d\Big(\sum_{i=3}^d \frac{1}{K^{i-2}}\sum_{j=0}^i|A_{i,j}|\Big) \leq \epsilon_0.
 \end{split}
 \end{equation}
 Let us show why the above inequality is true. By the definition of $A_{i,j}$ (see \eqref{322}), we have
 \begin{equation}
 \begin{split}
     &A_{2,0}=a_{2,0}+\sum_{i'=3}^{d}\sum_{j'=0}^{i'} \binom{i'-j'}{2}c_1^{i'-j'-2}c_2^{j'}a_{i',j'}
     \\&
     A_{2,2}=a_{2,2}+\sum_{i'=3}^{d}\sum_{j'=2}^{i'} \binom{j'}{2}c_1^{i'-j'}c_2^{j'-2}a_{i',j'}.
\end{split}
 \end{equation}
 Note that $A_1$  is a linear combination of $\{a_{i',j'}\}_{i' \geq 3} \cup \{a_{i',j'}\}_{j' \geq 3}$. The constants $A_{i,j}$ with $i \geq 3$ or $j \geq 3$ also satisfy the same property. Therefore, \eqref{326} holds true by the initial condition on the coefficients $a_{i,j}$. Hence, we can apply the induction hypothesis and obtain
 \begin{equation}
 \begin{split}
 \|E_Sf_{\tau}\|_{L^{3.25}(B_{R'})} &\lesssim
 K^{-2+\frac{4}{3.25}}\|E_{S_1}g\|_{L^{3.25}(B_{R'})}
 \\& \lesssim
 K^{-2+\frac{4}{3.25}}C_{\epsilon,d,\lambda}(R')^{\epsilon}\|g\|_{L^2}^{1-\lambda}\|g\|_{L^{\infty}}^{\lambda}
 \\& \lesssim
 K^{-1-\lambda+\frac{4}{3.25}}C_{\epsilon,d,\lambda}(R')^{\epsilon}\|f_{\tau}\|_{L^2}^{1-\lambda}\|f_{\tau}\|_{L^{\infty}}^{\lambda}.
 \end{split}
 \end{equation}
The last inequality follows from
 $
 \|g\|_{L^{\infty}}=\|f_{\tau}\|_{L^{\infty}}$ and
 $\|g\|_{L^2}=K\|f_{\tau}\|_{L^2}$. This completes the proof.
\end{proof}

\begin{lem}\label{striprescaling}
Under the induction hypothesis \eqref{induction},
for every $L \in \mathbb{L}$, large numbers $K_L,R'$ with $10^{10d} \leq K_L \leq R' \leq R/2$, $3/13<\lambda<1$, surface $S$ with a polynomial $h$ satisfying \eqref{maininequality}, ball $B_{R'}$ of radius $R'$, and function $f:[-1,1]^2 \rightarrow \C$, we have
\begin{equation}
\Norm{E_Sf_{L} }_{L^{3.25}(B_{
R'})}
\leq (K_L)^{(-\frac12-\frac{\lambda}{2}+\frac{2}{3.25})}
C_dC_{\epsilon,d,\lambda}
({R'})^{\epsilon}\|f_{L}\|_{L^2}^{1-\lambda}\|f_L\|_{L^\infty}^{\lambda},
\end{equation}
for some constant $C_d$ depending only on $d$. 
\end{lem}

\begin{proof}[Proof of Lemma \ref{striprescaling}]
The proof is very similar to that of Lemma \ref{squarerescaling}.
Without loss of generality, we assume that the center of the ball $B_{R'}$ is the origin.
Also, after applying an affine transformation, we assume that our strip $L$ is 
\begin{equation}
    [-1,1] \times [-c_1(K_L^{-1}), c_1(K_L^{-1})].
\end{equation}
Since $L \in \mathbb{L}$, by the definition of a bad line, we obtain
\begin{equation}
\max_{k=2,\ldots,d}\{ |a_{k,0}| \} \leq 10^{-5d}K_L^{-1}\epsilon_0.
\end{equation}
We apply the anisometric scaling $(\xi,\eta) \mapsto (\xi,K_L\eta)$. After this scaling, $E_Sf_L(x)$ becomes 
\begin{equation}
\begin{split}
K_L^{-1}\int_{[-1,1]^2}
f_L(\xi,{K_L^{-1} \eta })e\big(\big\langle
(x_1,{K_L^{-1}x_2},K_L^{-1}{x_3}),(\xi,\eta,h_2(\xi,\eta))
\big\rangle \big)
\,d\xi d\eta.
\end{split}
\end{equation}
Here, the polynomial $h_2(\xi,\eta)$ is given by
\begin{equation}
\begin{split}
\xi\eta+a_{2,0}K_L\xi^2+a_{2,2}K_L^{-1}\eta^2
+
\sum_{i=3}^{d}\sum_{j=0}^{i}a_{i,j}(K_L^{-1})^{j-1}\xi^{i-j}\eta^j.
\end{split}
\end{equation}
Let us use the notation $g(\xi,\eta)=f(\xi,K_L^{-1}\eta)$ and let $S_2$ denote the surface associated with $h_2$. We can write 
\begin{equation}
E_{S}f_L(x)=K_L^{-1}E_{S_2}g\Big(x_1,K_L^{-1}{x_2},K_L^{-1}{x_3}\Big).
\end{equation}
We apply the change of variables on the physical variables
$
(x_1,x_2,x_3) \mapsto (x_1,K_L^{-1}x_2,K_L^{-1}x_3)
$
and obtain 
\begin{equation}
    \Norm{E_Sf_{L} }_{L^{3.25}(B_{
R'})}\lesssim
(K_L)^{-1+\frac{2}{3.25}}\|E_{S_2}g\|_{L^p(B_{R'})}.
\end{equation}
The polynomial $h_2(\xi,\eta)$ may not satisfy the induction hypothesis \eqref{maininequality}. To deal with this issue, we decompose the square $[-1,1]^2$ into smaller squares with side length $10^{-10d}$. By the triangle inequality, we obtain
\begin{equation}
    \|E_{S_2}g\|_{L^p(B_{R'})} \lesssim \sum_{ q\in \mathcal{P}(10^{-10d})}\|E_{S_2}g_q\|_{L^p(B_{R'})} \lesssim 
    \sup_{ q \in \mathcal{P}( 10^{-10d})}\|E_{S_2}g_q\|_{L^p(B_{R'})}.
\end{equation}
For convenience, we assume that $q=[0,10^{-10d}]^2$. The proof of the general case is similar. We apply the scaling $(\xi,\eta) \mapsto (10^{10d}\xi,10^{10d}\eta)$. Then $E_{S_2}g_q(x)$ becomes
\begin{equation}
    10^{-20d}\int_{[-1,1]^2}g_q(10^{-10d}\xi,10^{-10d}\eta)
    e\big(\big\langle
(x_1',{x_2'},{x_3'}),(\xi,\eta,h_3(\xi,\eta))
\big\rangle \big)
\,d\xi d\eta
\end{equation}
where $(x_1',x_2',x_3')=(10^{-10d}x_1,10^{-10d}x_2,10^{-20d}x_3)$ and $h_3(\xi,\eta)$ is given by
\begin{equation}
\begin{split}
\xi\eta+a_{2,0}K_L\xi^2+a_{2,2}K_L^{-1}\eta^2
+
\sum_{i=3}^{d}10^{-10(i-2)d}\sum_{j=0}^{i}a_{i,j}(K_L^{-1})^{j-1}\xi^{i-j}\eta^j.
\end{split}
\end{equation}
This polynomial satisfies the induction assumption \eqref{maininequality}.
We
apply the change of variables 
\begin{equation}
(x_1,x_2,x_3) \mapsto (10^{-10d}x_1,10^{-10d}x_2,10^{-20d}x_3),
\end{equation}
and by
the induction hypothesis \eqref{induction} we obtain
\begin{equation}
\begin{split}
\Norm{E_Sf_{L} }_{L^{3.25}(B_{
R'})}&\lesssim
(K_L)^{-1+\frac{2}{3.25}}\|E_{S_2}g\|_{L^p(B_{R'})}
\\& \lesssim C_d C_{\epsilon,d,\lambda}(K_L)^{-1+\frac{2}{3.25}}(R')^{\epsilon}\|g\|_{L^2}^{1-\lambda}\|g\|_{L^{\infty}}^{\lambda}
\\&\lesssim C_d C_{\epsilon,d,\lambda}(K_L)^{(-\frac12-\frac{\lambda}{2}+\frac{2}{3.25})}(R')^{\epsilon}\|f_{L}\|_{2}^{1-\lambda}\|f_{L}\|_{\infty}^{\lambda}.
\end{split}
\end{equation}
The last inequality follows from $\|g\|_{{\infty}}=\|f_{L}\|_{{\infty}}$ and $\|g\|_{2}=(K_L)^{\frac12}\|f_L\|_{2}$.
This completes the proof of Lemma \ref{striprescaling}.
\end{proof}

\subsection{Bad pairs}\label{0522subsection2.2}
Let $\zeta$ be a point on $[-1, 1]^2$.
We define an affine transformation $T_{\zeta}$ associated to $\zeta$ in the following way. First, let $T_{\zeta, 1}$ be the action of translation that sends $\zeta$ to the origin. Consider the resulting polynomial 
\begin{equation}
(h \circ T_{\zeta, 1}^{-1})(\xi,\eta)=\sum_{i=0}^{d}\sum_{j=0}^{i}a'_{i,j}\xi^{i-j}\eta^j, 
\end{equation}
for some new coefficient $a'_{i, j}$. We examine its quadratic terms. Via simple computation, it is not difficult to see that 
\begin{equation}
    |a'_{2, 0}|+ |a'_{2, 2}|\lesim \epsilon_0, \ \  |a'_{2, 1}-1| \lesim \epsilon_0.
\end{equation}
Therefore, one can find a linear transformation, called $T_{\zeta, 2}$, which is a small perturbation of the identity map, such that 
\begin{equation}
(h \circ T_{\zeta,1}^{-1}\circ T_{\zeta, 2}^{-1})(\xi,\eta)=\sum_{i=0}^{d}\sum_{j=0}^{i}a''_{i,j}\xi^{i-j}\eta^j, \;\; a_{2,1}''=1,\;\; a''_{2,0}=a''_{2,2}=0
\end{equation}
for some coefficients $a''_{i,j}$. Define 
\begin{equation}\label{0422e2.12}
T_{\zeta}:=T_{\zeta, 2}\circ T_{\zeta, 1}. 
\end{equation}
We will use the map $T_{\zeta}$ frequently throughout the paper. {Define a polynomial $h_{\zeta}$, associated to the point $\zeta$, by }
\beq\label{noqudratic}
h_{\zeta}(\xi,\eta):=\xi\eta+\sum_{i=3}^{d}\sum_{j=0}^i a''_{i,j}\xi^{i-j}\eta^j.
\endeq

Now we consider two points $\zeta_1, \zeta_2$ with $|\zeta_1-\zeta_2|>K_L^{-1}$. Denote by $\overline{\zeta_2}=(\overline{\zeta_{2,1}}, \overline{\zeta_{2,2}})\in \R^2$ the image of $\zeta_2$ under the map $T_{\zeta_1}$. Note that $|\overline{\zeta_2}|>(2K_L)^{-1}$. Let us consider the case that $|\overline{\zeta_{2,1}}|\ge |\overline{\zeta_{2,2}}|$. Under this assumption, it is easy to see that 
\beq\label{secondderivativelowerbound}
|\partial_{2} h_{\zeta_1}(\overline{\zeta_{2}})| \geq \frac{1}{3K_L} \text{ \, and \, } |\partial_{2} h_{\zeta_1}(\overline{\zeta_{2}})|\ge
\frac12{|\partial_{1} h_{\zeta_1}(\overline{\zeta_{2}})|}. 
\endeq 
Here we have used the separation assumption on $\zeta_1$ and $\zeta_2$. 
We define an auxiliary function associated to the pair $(\zeta_1,\zeta_2)$ by
\beq\label{badpairvariety}
P_{\zeta_1,\zeta_2}(\xi,\eta):=
\begin{vmatrix}
\partial_1h_{\zeta_1}(\xi,\eta) & \partial_2h_{\zeta_1}(\xi,\eta)
\\
\partial_1h_{\zeta_1}(\overline{\zeta_{2}}) & \partial_2h_{\zeta_1}(\overline{\zeta_{2}})
\end{vmatrix}.
\endeq
Note that $P_{\zeta_1,\zeta_2}(0)=P_{\zeta_1,\zeta_2}(\overline{\zeta_2})=0$.
Note also that
\begin{equation}\label{badpairvariety2}
    P_{\zeta_1,\zeta_2}(\xi,\eta)=
\begin{vmatrix}
\partial_1h_{\zeta_1}(T_{\zeta_1}(\zeta_1)) & \partial_2h_{\zeta_1}(T_{\zeta_1}(\zeta_1)) & -1
\\
\partial_1h_{\zeta_1}(T_{\zeta_1}(\zeta_2)) & \partial_2h_{\zeta_1}(T_{\zeta_1}(\zeta_2)) & -1
\\
\partial_1h_{\zeta_1}(\xi,\eta) & \partial_2h_{\zeta_1}(\xi,\eta) & -1
\end{vmatrix}.
\end{equation}
Set the zero set of this polynomial to be 
\begin{equation}
    \mc{Z}_{\zeta_1, \zeta_2}:=\{(\xi, \eta): P_{\zeta_1, \zeta_2}(\xi, \eta)=0\}.
\end{equation}
Notice that $\mc{Z}_{\zeta_1, \zeta_2}$ contains both the origin and $\overline{\zeta_2}$. From \eqref{secondderivativelowerbound}, we see that 
\begin{equation}\label{0422e2.18}
    |\partial_2 P_{\zeta_1,\zeta_2}(\xi,\eta)| \ge \frac{1}{4K_L} \text{ \,  and \,  }
    |\partial_2 P_{\zeta_1,\zeta_2}(\xi,\eta)|\ge \frac13{|\partial_1 P_{\zeta_1,\zeta_2}(\xi,\eta)|},
\end{equation}
for every $(\xi,\eta) \in [-4,4]^2$. Let us give more details on the second inequality. By the definition \eqref{badpairvariety},
\begin{equation}
    \begin{split}
        &\partial_2P_{\zeta_1,\zeta_2}(\xi,\eta)=\partial_{12}h_{\zeta_1}(\xi,\eta)\partial_2h_{\zeta_1}(\overline{\zeta_2})-\partial_{22}h_{\zeta_1}(\xi,\eta)\partial_1h_{\zeta_1}(\overline{\zeta_2})
        \\&\partial_1P_{\zeta_1,\zeta_2}(\xi,\eta)=\partial_{11}h_{\zeta_1}(\xi,\eta)\partial_2h_{\zeta_1}(\overline{\zeta_2})-\partial_{12}h_{\zeta_1}(\xi,\eta)\partial_{1}h_{\zeta_1}(\overline{\zeta_2}).
    \end{split}
\end{equation}
The second inequality of \eqref{0422e2.18} now follows by \eqref{secondderivativelowerbound}.
By the implicit function theorem, we can write $\mc{Z}_{\zeta_1, \zeta_2}\cap [-4, 4]^2$ as 
\begin{equation}
    \{(\xi, \eta_{{\zeta_1,\zeta_2}}(\xi)): \xi\in [-4, 4]\},
\end{equation}
for some smooth function $\eta_{\zeta_1,\zeta_2}$ with $\eta_{\zeta_1,\zeta_2}(0)=0$ and $\eta_{\zeta_1,\zeta_2}(\overline{\zeta_{2,1}})=\overline{\zeta_{2,2}}$. Moreover, from \eqref{0422e2.18} we can conclude that\footnote{We postpone the details to \eqref{dervativeofeta} in order not to interrupt the discussion.}
\begin{equation}
    |\eta'_{\zeta_1,\zeta_2}(\xi)|\lesssim 1, \text{ for every } \xi. 
\end{equation}
Define the function $H_{\zeta_1,\zeta_2}(\xi)$ by
\begin{equation}
H_{\zeta_1, \zeta_2}(\xi):=h_{\zeta_1}(\xi,\eta_{\zeta_1,\zeta_2}(\xi))
.
\end{equation}
Define two bad sets
\begin{equation}\label{353}
    \begin{split}
        &\mathbb{B}_1(\zeta_1,\zeta_2):=\{(\xi,\eta_{\zeta_1,\zeta_2}(\xi))\in [-4,4]^2 : |\eta_{\zeta_1,\zeta_2}'(\xi)-\eta_{\zeta_1,\zeta_2}'(0)| \leq K_1^{-1}  \},
    \\
    &\mathbb{B}_2(\zeta_1,\zeta_2):=\{(\xi,\eta_{\zeta_1,\zeta_2}(\xi))\in [-4,4]^2:  |H'_{\zeta_1,\zeta_2}(\xi)| \leq K_1^{-1}  \}.
    \end{split}
\end{equation}
Now we have all the notions we need to define bad pairs. We say that $(\zeta_1,\zeta_2)$ is \textit{a bad pair}
if 
\begin{equation}
    \overline{\zeta_2} \in \mathbb{B}(\zeta_1,\zeta_2),
\end{equation}
where 
\begin{equation}
    \mathbb{B}(\zeta_1, \zeta_2):=\mathbb{B}_1(\zeta_1,\zeta_2) \cap \mathbb{B}_2(\zeta_1,\zeta_2).
\end{equation}
Otherwise we say that $(\zeta_1,\zeta_2)$ is \textit{a good pair}. Define a bad pair and a good pair similarly in the case that $|\overline{\zeta_{2,1}}|\leq |\overline{\zeta_{2,2}}|$.  Note also that even if $(\zeta_1,\zeta_2)$ is a good pair, according to our definition, $(\zeta_2,\zeta_1)$ may not be a good pair. 

Suppose that $\tau_1,\tau_2$ are squares with side length $K^{-1}$ and $\mathrm{dist}(\tau_1,\tau_2)\geq 2K_L^{-1}$. We say that $(\tau_1,\tau_2)$ is a bad pair if $(\zeta_1,\zeta_2)$ is a bad pair for every $\zeta_i \in 2\tau_i$. Otherwise, we say that $(\tau_1,\tau_2)$ is a good pair. 
\\

The total number of bad pairs could be enormous. However, under some circumstance, we are able to describe their distribution. The following lemma plays a crucial role in the proof of Theorem \ref{broadfunctionestimate}.

\begin{lem}\label{badpairlemma}
Fix sufficiently small $\epsilon_0>0$ and $\epsilon>0$. Let $\zeta_1,\zeta_2$ be points in $[-1,1]^2$ with $|\zeta_1-\zeta_2| >K_L^{-1}$.  
For sufficiently large dyadic numbers $K_L, K_{d},\ldots,K_1,K$ with 
\begin{equation}\label{0522e2.56}
    K_L=:K_{d+1} \ll K_{d} \ll \cdots \ll  K_1 \ll K_0 :=K
\end{equation}
the following holds true: Either there exist squares $\upsilon_{j,i} \in \mathcal{P}(K_j^{-1},[-2,2]^2)$ such that
    \begin{equation}
    \mathbb{B}(\zeta_1,\zeta_2)
    \subset 
    \Big(
    \bigcup_{i=1}^{d^2}T_{\zeta_1}(\upsilon_{d+1,i})\Big) \cup
    \Big(\bigcup_{j=2}^{d-1} \bigcup_{i=1}^{O(K_{j+1}^{3}) }T_{\zeta_1}(\upsilon_{j,i})\Big) 
    \end{equation}
or there exist bad strips $L_j$ such that
\begin{equation}
\mathcal{N}_{K_L^{-2}}\big(\mathbb{B}(\zeta_1,\zeta_2) \big) \subset \bigcup_{j=1}^{2} T_{\zeta_1}(L_j).
\end{equation}
The squares $\upsilon_{j,i}$ and the strips $L_j$ depend on a choice of  $\zeta_1$ and $\zeta_2$.
\end{lem}

Note that for each $j$ the number of the squares $\upsilon_{j,i}$ is $O(K_{j+1}^3)$. Even though $K_{j+1}^3$ is a large number compared to one, it is still very small compared to $K_j$. Hence, the number of squares will not cause any trouble. 

To prove this lemma, we will consider several cases. In some cases, we will reduce our problem to a problem involving a one-variable polynomial and apply Lemma \ref{onepolynomiallemma}.
Recall that for a polynomial $p(\xi)$, $\|p\|$ denote the supremum of all its coefficients.

\begin{lem}\label{onepolynomiallemma}
Let $n \geq 1$.
For every polynomial
\begin{equation}
p(\xi)=b_0+b_1\xi+\cdots+b_n\xi^n
\end{equation}
with $\|p\|=1$
and
for every $0< \gamma \lesim 1$, there exist points $\{\xi_i\}_{i=1}^{{n(n+1)}/{2}}$ such that
\begin{equation}\label{2.26}
\big\{\xi \in [-4,4]:|p(\xi)|<\gamma\big\} \subset \bigcup_{i=1}^{\frac{n(n+1)}{2}} \Big([\xi_i-4\gamma^{\frac{1}{2^n}},\xi_i+4\gamma^{\frac{1}{2^n}}]\Big).
\end{equation} 
\end{lem}

This lemma is a corollary of Lemma 2.14 in \cite{arxiv:1902.03450}. We will give a proof for the completeness of the paper.

\subsubsection*{Proof of Lemma \ref{onepolynomiallemma}}
Let us use the notation $I=[-4,4]$ for simplicity.
We claim that for every $\sigma>0$ and polynomial $p_0$ with $\|p_0''\|_{L^{\infty}(I)} \leq 1$  the following holds true:
\begin{equation}\label{2.27}
\begin{split}
& \{\xi \in I:|p_0(\xi)| \leq \sigma^2\} \subset \{\xi \in I :|p_0'(\xi)| \leq  2\sigma\} \cup \mc{N}_{\sigma} (\mc{Z}(p_0)),
\end{split}
\end{equation}
where $\mc{Z}({p_0})$ denotes the zero set of $p_0$. We assume that this claim holds true for a moment and deduce \eqref{2.26}.
Note that if  $\|p' \| \ll 1$, then since the constant term is $\gtrsim 1$, the left hand side of \eqref{2.26} is empty for sufficiently small $\gamma>0$. Hence, we may assume that $\|p'\| \simeq 1$.
By the fundamental theorem of algebra, it is not difficult to see that
the lemma follows by applying the claim repeatedly.

It remains to prove this claim.  First note that if $|p_0(\xi)| \leq \sigma^2$ and $|p_0'(\xi)| > 2\sigma$ then
\beq\label{taylorinequality}
|p_0(\xi)|-\sigma|p_0'(\xi)|+\frac12{\sigma^2\|p_0''\|_{L^{\infty}(I)}} < 0.
\endeq
Our goal to show that in this case $\mathrm{dist}(\xi, \mc{Z}(p_0)) \leq \sigma$.
Without lose of generality, we may assume that $p_0(\xi)>0$.
Suppose that $\mathrm{dist}(\xi, \mc{Z}(p_0)) > \sigma$ for a moment and deduce the contradiction. We consider the case that $p_0'(\xi)>0$. By Taylor's theorem,
\begin{equation}
0< p_0(\xi-\sigma) \leq p_0(\xi)-\sigma p_0'(\xi) +\frac12{\sigma^2 \|p_0''\|_{L^{\infty}(I)}}.
\end{equation}
This contradicts with \eqref{taylorinequality}. We consider the other case that $p_0'(\xi)<0$. By Taylor's theorem,
\begin{equation}
0< p_0(\xi+\sigma) \leq p_0(\xi)+\sigma p_0'(\xi) +\frac12{\sigma^2 \|p_0''\|_{L^{\infty}(I)}}.
\end{equation}
This also contradicts \eqref{taylorinequality}, and therefore completes the proof of Lemma \ref{onepolynomiallemma}.
\qed

\medskip

\subsubsection*{Proof of Lemma \ref{badpairlemma}}
{Recall the map $T_{\zeta_1}$ from \eqref{0422e2.12}. Denote by $\overline{\zeta_2}=(\overline{\zeta_{2,1}}, \overline{\zeta_{2,2}})$ the image of ${\zeta_2}$ under $T_{\zeta_1}$. Without loss of generality, we assume that $|\overline{\zeta_{2,1}}|\ge |\overline{\zeta_{2,2}}|$. }
Throughout the proof, for the sake of simplicity, $h_{\zeta_1}$ will be abbreviated to $h_0$, $H_{\zeta_1, \zeta_2}$  simply to $H$, $\eta_{\zeta_1,\zeta_2}$ to $\eta_0$, and $\mathbb{B}(\zeta_1, \zeta_2)$ to $\mathbb{B}$. 

Let us first compute the derivative of $\eta_{0}(\xi)$.
Recall that the function $\eta_{0}(\xi)$ is implicitly defined by
\beq\label{varietyequality}
P_{\zeta_1,\zeta_2}(\xi,\eta_0(\xi))=\partial_1 h_0(\xi,\eta_0(\xi))\partial_2 h_0(\overline{\zeta_2})-\partial_2 h_0(\xi,\eta_0(\xi))\partial_1 h_0(\overline{\zeta_2})=0.
\endeq
We take a derivative with respect to $\xi$ and obtain
\begin{equation}
\begin{split}
    &\partial_{11}h_0(\xi,\eta_0(\xi))\partial_2h_0(\overline{\zeta_2})+
    \partial_{12}h_0(\xi,\eta_0(\xi))
    \eta_0'(\xi)
    \partial_2h_0(\overline{\zeta_2})
    \\&
    -\partial_{12}h_0(\xi,\eta_0(\xi))\partial_1h_0(\overline{\zeta_2})-\partial_{22}h_0(\xi,\eta_0(\xi))\eta_0'(\xi) \partial_1h_0(\overline{\zeta_2})=0.
\end{split}
\end{equation}
By rearranging terms, we obtain
\beq\label{derivativeofeta}
\eta_0'(\xi)=\frac{\partial_{12}h_0(\xi,\eta_0(\xi))\partial_1h_0(\overline{\zeta_2})-\partial_{11}h_0(\xi,\eta_0(\xi))\partial_2h_0(\overline{\zeta_2}) }{\partial_{12}h_0(\xi,\eta_0(\xi))\partial_2h_0(\overline{\zeta_2})-\partial_{22}h_0(\xi,\eta_0(\xi))\partial_1h_0(\overline{\zeta_2})}.
\endeq
Note that 
\begin{equation}\label{0423e2.28}
\eta_0'(0)=\frac{\partial_1h_0(\overline{\zeta_2})}{\partial_2h_0(\overline{\zeta_2})}.
\end{equation}
Moreover, by taking derivatives repeatedly to \eqref{varietyequality}, it is not difficult to show that
\begin{equation}\label{dervativeofeta}
    |\eta_0^{(j)}(\xi)| \lesssim 1
    \text{ for every } \xi \in [-4,4] 
\end{equation}
for every $j=1,\ldots,d$.
By the chain rule and \eqref{varietyequality}, we obtain
\begin{equation}\label{0422e2.28}
\begin{split}
|H'(\xi)|&=|\partial_1h_0(\xi,\eta_0(\xi))+\partial_2h_0(\xi,\eta_0(\xi))\eta_0'(\xi)|
\\& 
=
\Big|\partial_2h_0(\xi,\eta_0(\xi))\big(\eta_0'(\xi)+\frac{\partial_1h_0(\overline{\zeta_2})}{\partial_2h_0(\overline{\zeta_2})}\big)\Big|.
\end{split}
\end{equation}
For $(\xi,\eta_0(\xi)) \in \mathbb{B}$, we apply the triangle inequality and obtain 
\beq\label{h'lowerbound}
\begin{split}
|H'(\xi)|
&\geq \Big|\partial_2h_0(\xi,\eta_0(\xi))\big(\eta_0'(0)+\frac{\partial_1h_0(\overline{\zeta_2})}{\partial_2h_0(\overline{\zeta_2})}\big)\Big|-K_1^{-1}
\\&
=
2\Big|\partial_2h_0(\xi,\eta_0(\xi))\big( \frac{\partial_1h_0(\overline{\zeta_2})}{\partial_2h_0(\overline{\zeta_2})}\big) \Big|-K_1^{-1}.
\end{split}
\endeq
Depending on the size of $\big|\frac{\partial_1h_0(\overline{\zeta_2})}{\partial_2h_0(\overline{\zeta_2})}\big|$, we have two cases
\begin{equation}\label{0422_case_distinction}
  \Big|\frac{\partial_1h_0(\overline{\zeta_2})}{\partial_2h_0(\overline{\zeta_2})}\Big| > {K_1^{-1/2}} \text{\; or\; }  \Big|\frac{\partial_1h_0(\overline{\zeta_2})}{\partial_2h_0(\overline{\zeta_2})}\Big| \le {K_1^{-1/2}}.
\end{equation}
The quantity that appears in the above display has a geometric meaning, see \eqref{0423e2.28}. 

\medskip

\subsubsection*{Case 1.} $\Big|\frac{\partial_1h_0(\overline{\zeta_2})}{\partial_2h_0(\overline{\zeta_2})}\Big| > {K_1^{-1/2}}$.

In this case, by \eqref{353} and \eqref{h'lowerbound}, for $(\xi,\eta_0(\xi)) \in \mathbb{B}$ we obtain
\begin{equation}
2K_1^{-1/2}|\partial_2 h_0(\xi,\eta_0(\xi))| - K_1^{-1}
 < |H'(\xi)| \leq K_1^{-1}.
\end{equation}
Multiplying $K_1^{1/2}$ on both sides we obtain
\begin{equation}
|\partial_2h_0(\xi,\eta_0(\xi))| < K_1^{-1/2},
\end{equation}
which certainly means that 
\begin{equation}\label{2.37}
\mathbb{B} \subset
\{(\xi,\eta_0(\xi)) \in \mathbb{B}: |\partial_2h_0(\xi,\eta_0(\xi))|<K_1^{-1/2} \}.
\end{equation}
By the lower bound \eqref{secondderivativelowerbound} and the equality \eqref{varietyequality}, for every  $(\xi,\eta_0(\xi)) \in \mathbb{B}$, we obtain 
\begin{equation}
    |\partial_1h_0(\xi,\eta_0(\xi))|<2 K_1^{-1/2}.
\end{equation}
By choosing $K_2 \ll K_1$ we can take $\upsilon \in \mathcal{P}(K_2^{-1},\R^2)$ such that $T_{\zeta_1}(\upsilon)$ contains the set on the right hand side of \eqref{2.37}.
This completes the discussion of Case 1.

\subsubsection*{Case 2.}
$\Big|\frac{\partial_1h_0(\overline{\zeta_2})}{\partial_2h_0(\overline{\zeta_2})}\Big| \le {K_1^{-1/2}}$.

In this case, by \eqref{varietyequality} and the trivial bound $|\partial_2h_0(\xi,\eta_0(\xi))|\leq 2$,
for every $(\xi,\eta_0(\xi))$, we obtain 
\begin{equation}
|\partial_1h_0(\xi,\eta_0(\xi))| \leq  2K_1^{-1/2}.
\end{equation}
Since $(\xi,\eta_0(\xi)) \in \mathbb{B}$, we obtain from  \eqref{0423e2.28}, our assumption in Case 2, and the triangle inequality that  \begin{equation}
|\eta'_0(\xi)| \leq 2 K_1^{-1/2}.
\end{equation}
This, combined with  \eqref{derivativeofeta}, implies that 
\begin{equation}
    |\partial_{11}h_0(\xi,\eta_0(\xi))|\leq  4 K_1^{-1/2} \leq K_2^{-1}.
\end{equation}
We have shown that
\begin{equation}\label{d1d2}
\begin{split}
    \mathbb{B} \subset 
    \big\{&(\xi,\eta_0(\xi))\in \mathbb{B} :|\partial_1h_0(\xi,\eta_0(\xi))| \leq  2K_1^{-1/2} \big\}
    \\& \cap
    \big\{(\xi,\eta_0(\xi)) \in \mathbb{B}:|\partial_{11}h_0(\xi,\eta_0(\xi))| \leq K_2^{-1} \big\}=:\mathcal{D}_1 \cap \mathcal{D}_2.
\end{split}
\end{equation}
 Depending on the size of $\max_{i=3,\ldots,d}\big({|a_{i,0}''|}\big)$, we have two subcases
\beq\label{424subcase1}
\max_{i=3,\ldots,d}\big({|a_{i,0}''|}\big) \leq
K_L^{-2} \text{    or    }
\max_{i=3,\ldots,d}\big({|a_{i,0}''|}\big) \geq
K_L^{-2}.
\endeq

\subsubsection*{Case 2.1}
$\max_{i=3,\ldots,d}\big({|a_{i,0}''|}\big) \leq
K_L^{-2}$. 

In this case, we have
\begin{equation}\label{2.45}
\sup_{\xi \in [-1,1]}|\partial_1h_0(\xi,0)| \leq d^2K_L^{-2}
\endeq
and since $\partial_1h_0(\xi,\eta)=\eta \cdot (1+O(\epsilon_0))+\partial_1h_0(\xi,0)$  we see that
\begin{equation}\label{2.85}
    \D_1 \subset \CN_{2d^2K_L^{-2}}(\{\eta=0\}).
\end{equation}
Since we are in the former case of \eqref{424subcase1} and the width of a bad strip is $2 \cdot 10^{-10d}\epsilon_0K_L^{-1}$, we can find two bad strips $L_1$ and $L_2$ so that the image of their union under the map $T_{\zeta_1}$ covers the set on the right hand side of \eqref{2.85}. By \eqref{d1d2}, we obtain
\begin{equation}
\mathbb{B} \subset T_{\zeta_1}(L_1) \cup T_{\zeta_1}(L_2).
\end{equation}
This gives the desired result. 

\subsubsection*{Case 2.2.}
$\max_{i=3,\ldots,d}\big({|a_{i,0}''|}\big) \geq
K_L^{-2}$.

Define the sets
\begin{equation}
    \D_j:=\{(\xi,\eta_0(\xi))\in \mathbb{B}: \Big|\frac{\partial^j h_0}{\partial \xi^j}(\xi,\eta_0(\xi))\Big| \leq K_j^{-1}   \}
\end{equation}
for every $j=3,\ldots,d$. Recall that $\D_1$ and $\D_2$ are defined in \eqref{d1d2}. We claim that
\begin{equation}\label{forwarditeration}
    \mathbb{B} \subset \Big(\bigcap_{j=1}^{d}\D_j \Big) \cup
    \Big( \bigcup_{j=2}^{d-1}\bigcup_{i=1}^{O(K_{j+1}^{2})}T_{\zeta_1}(\upsilon_{j,i})   \Big)
\end{equation}
for some $\upsilon_{j,i} \in \mathcal{P}(K_j^{-1},\R^2)$.

Let us prove the claim inductively. Suppose that we have obtained that 
\begin{equation}\label{2.49}
    \mathbb{B} \subset \Big(\bigcap_{j=1}^{k}\D_j \Big) \cup
    \Big( \bigcup_{j=2}^{k-1}\bigcup_{i=1}^{O(K_{j+1}^{2})}T_{\zeta_1}(\upsilon_{j,i})   \Big)
\end{equation}
for some $k \in \{2,\ldots,d-1\}$. Note that the case $k=2$ follows from \eqref{d1d2}.
Note first that
\begin{equation}
\bigcap_{j=1}^{k} \D_j \subset \big( \bigcap_{j=1}^{k+1} \D_j \big) \cup \big(\bigcap_{j=1}^{k} \D_j  \cap (\D_{k+1})^c  \big).
\end{equation}
We will show that
\begin{equation}\label{d1d2d3}
    \bigcap_{j=1}^{k} \D_j  \cap (\D_{k+1})^c
    \subset \bigcup_{i=1}^{O(K_{k+1}^{2})}T_{\zeta_1}(\upsilon_{k,i}).
\end{equation}
This gives the proof of \eqref{2.49} for $k+1$ and by an inductive argument it finishes the proof of the claim \eqref{forwarditeration}.

Let us prove \eqref{d1d2d3}.
Let $(\xi_0,\eta_0(\xi_0))$ be a point belonging to the set on the left hand side of \eqref{d1d2d3}. It is not difficult to see that \eqref{d1d2d3} follows from the following geometric observation: All points $(\xi,\eta_0(\xi)) \in \mathbb{B}$ with 
\begin{equation}
    CK_{k+1}K_{k}^{-1} \leq |\xi-\xi_0| \leq C' K_{k+1}^{-1}
\end{equation}
do not belong to $\bigcap_{j=1}^{k} \D_j  \cap (\D_{k+1})^c$. Here $C, C'$ are two constant that are to be chosen. Let us give more details on why the geometric observation implies \eqref{d1d2d3}. We take $K_{k+1}^{-1}$-separated points on the graph $(\xi,\eta_0(\xi))$. On each point, we apply the geometric observation, and take $O(K_{k+1}^{-2})$ many $K_k^{-1}$-squares inside the ball of radius $CK_{k+1}K_k^{-1}$ centered at the point. This gives \eqref{d1d2d3}. Let us now show the observation. Divide \eqref{varietyequality} by $\partial_2h_0(\overline{\zeta_2})$. By taking derivatives with respect to $\xi$, for every $j \geq 2$, we obtain
\begin{equation}\label{392}
    \frac{\partial^{j-1}}{\partial \xi^{j-1}}\Big( \partial_1 h_0(\xi,\eta_0(\xi))\Big) = 
    \frac{\partial^{j-1}}{\partial \xi^{j-1}}\Big( \partial_2 h_0(\xi,\eta_0(\xi))\Big)\frac{\partial_1h_0(\overline{\zeta_2}) }{\partial_2h_0(\overline{\zeta_2}) }.
\end{equation}
Since we are considering the latter case of \eqref{0422_case_distinction}, the right hand side is bounded by $O(K_1^{-1/2})$. By applying the chain rule to \eqref{392}, we obtain
\begin{equation}
\begin{split}
\frac{\partial^j h_0}{\partial \xi^j}(\xi,\eta_0(\xi))&+\frac{\partial^2 h_0}{\partial \xi \partial \eta}(\xi,\eta_0(\xi))\eta^{(j-1)}_0(\xi)
\\&
+\mathrm{Error}(\xi,\eta_0^{(0)}(\xi),\ldots,\eta_0^{(j-2)}(\xi))=O(K_1^{-1/2}).
\end{split}
\end{equation}
where
\begin{equation}
    |\mathrm{Error}(\xi,\eta_0^{(0)}(\xi),\ldots,\eta_0^{(j-2)}(\xi))| \lesssim \sum_{l=1}^{j-2}|\eta_0^{(l)}(\xi)|.
\end{equation}
Hence, by an inductive argument, we see that
\begin{equation}\label{2.52}
|\eta_0^{(j-1)}(\xi_0)| \simeq
    \Big|\frac{\partial^j h_0}{\partial \xi^j}(\xi_0,\eta_0(\xi_0))\Big|  \leq  K_{j}^{-1}
\end{equation}
and 
\begin{equation}\label{2.96}
|\eta^{(k)}_0(\xi_0)| \simeq
    \Big|\frac{\partial^{k+1} h_0}{\partial \xi^{k+1}}(\xi_0,\eta_0(\xi_0))\Big| > K_{k+1}^{-1}.
\end{equation}
for every  $j=2,\ldots,k$.
By \eqref{2.52} and the previous discussion, it suffices to show that for every point $(\xi,\eta_0(\xi)) \in \mathbb{B}$ with $CK_{k+1}K_k^{-1} \leq |\xi-\xi_0| \leq C'K_{k+1}^{-1}$ it holds that
\begin{equation}\label{2.54}
    |\eta_0^{(k-1)}(\xi)| >CK_k^{-1}.
\end{equation}
Let us prove \eqref{2.54}. By the fundamental theorem of calculus and \eqref{2.52} with $j=k$, we obtain
\begin{equation}\label{secondderivativeofeta}
\begin{split}
    \eta_0^{(k-1)}(\xi)&=\int_{\xi_0}^{\xi}\eta_0^{(k)}(x)\,dx+\eta_0^{(k-1)}(\xi_0)
    \\&
    =\int_{\xi_0}^{\xi}\big(\int_{\xi_0}^{x}\eta^{(k+1)}_0(y)\,dy+\eta^{(k)}_0(\xi_0) \big)\,dx+\eta_0^{(k-1)}(\xi_0)
    \\&=
    \int_{\xi_0}^{\xi}\int_{\xi_0}^{x}\eta_0^{(k+1)}(y)dydx+(\xi-\xi_0)\eta_0^{(k)}(\xi_0)+O(K_k^{-1})
    \\&=
    \int_{\xi_0}^{\xi}(\xi-y)\eta_0^{(k+1)}(y)dy+(\xi-\xi_0)\eta_0^{(k)}(\xi_0)+O(K_k^{-1}).
\end{split}
\end{equation}
We use   \eqref{2.96}, \eqref{secondderivativeofeta}, and a trivial bound \eqref{dervativeofeta}, and obtain
\begin{equation}
    |\eta_0^{(k-1)}(\xi)| \gtrsim 
    K_{k+1}^{-1}|\xi-\xi_0|-c|\xi-\xi_0|^2-cK_k^{-1},
\end{equation}
for some small constant $c$.  It is not difficult to see that if $4cK_{k+1}K_{k}^{-1} \leq |\xi-\xi_0| \leq C'K_{k+1}^{-1}$ then the above lower bound is greater than $2cK_k^{-1}$ and this finishes the proof of \eqref{2.54}, and thus \eqref{2.49} and \eqref{forwarditeration}.
\medskip

It remains to prove that
\begin{equation}\label{backwarditeration}
    \bigcap_{j=1}^{d}\D_j \subset
    \bigcup_{i=1}^{d^2}T_{\zeta_1}(\upsilon_{d+1,i}) 
\end{equation}
for some $\upsilon_{d+1,i} \in \mathcal{P}(K_{d+1}^{-1},\R^2)$. For simplicity, we first write
\begin{equation}\label{uniformconstant}
    \bigcap_{j=1}^{d}\D_j \subset \Big\{(\xi,\eta) \in [-4,4]^2: \max_{j=1,\ldots,d}\Big|\frac{\partial^j h_0}{\partial \xi^j}(\xi,\eta)\Big| \leq K_{d}^{-1}  \Big\}.
\end{equation}
By Taylor's theorem with respect to $\xi$-variable, we obtain
\begin{equation}
    \frac{\partial h_{0}}{\partial \xi}(0,\eta)=
    \sum_{j=1}^{d}(-1)^{j-1}
    \frac{ \xi^{j-1}}{(j-1)!}
    \frac{\partial^j h_0}{\partial \xi^j}(\xi,\eta).
\end{equation}
Thus, the set on the right hand side of \eqref{uniformconstant} is contained in
\begin{equation}
    \Big\{(\xi,\eta) \in [-4,4]^2:
    \max\Big(\Big|\frac{\partial h_0}{\partial \xi}(\xi,\eta)\Big|
    ,\Big|\frac{\partial h_0}{\partial \xi}(0,\eta)\Big|
    \Big) \lesssim K_{d}^{-1}  \Big\}.
\end{equation}
Since $\frac{\partial h_0}{\partial \xi}(0,\eta)=\eta \cdot(1+ O(\epsilon_0))$, if a point $(\xi,\eta)$ belongs to the above set, we obtain $|\eta| \lesssim K_d^{-1}$. Hence, by treating $\eta$ as an error, the above set is contained in
\begin{equation}
    \Big\{(\xi,\eta) \in [-4,4]^2: |\eta| \lesssim K_d^{-1},
    \Big|\frac{\partial h_0}{\partial \xi}(\xi,0)\Big|
    \lesssim K_{d}^{-1}  \Big\}.
\end{equation}
Since we are in the latter case of \eqref{424subcase1}, the norm of the polynomial $\partial_1h_0(\xi,0)$ has a size $\gtrsim K_L^{-2}$ and we apply Lemma \ref{onepolynomiallemma} and obtain \eqref{backwarditeration}. This completes the proof. \qed

\section{Basic setup: Broad function and wave packet decomposition}
 In this section, we  follow the broad-narrow analysis of Bourgain and Guth \cite{MR2860188} and Guth \cite{MR3454378}, and introduce the notation of broad points. One advantage of introducing this notion is that with it we can start making use of multilinear restriction estimates (bilinear in our case). 
 
 Our notion of broad points is slightly more complicated than those in \cite{MR3454378, MR3653943}, see \eqref{0511e3.2} and \eqref{0507e3.3}. First of all, it involves many different scale, including $K_{d+1},\cdots,  K_0$. In other words, we need to make sure that our function $|Ef(x)|$ is ``broad" at every scale $K_j$. This is necessary as in Lemma \ref{badpairlemma} bad pairs appear at every scale. 
Next, the notion of broad points involves bad strips. A similar notion was already used in \cite{MR3653943}. We need to get rid of bad strips as it will be difficult to obtain a bilinear restriction estimate as good as the one in \cite{MR3454378}, if certain main contributions come from a bad strip.
However, a technical issue arises here. In \cite{MR3653943}, bad strips are either horizontal or vertical. Two bad strips are either disjoint or intersect at a square. In our case, orientations of bad strips are more complicated and we need to take into account all the possible intersections of bad strips.

We also review wave packet decomposition in this section.

\subsection{Broad function}
Take dyadic numbers
\begin{equation}
    K_L=K_{d+1} \ll K_d \ll \cdots \ll  K_1 \ll K_0=K.
\end{equation}
Let $M:=10^{20d}\epsilon_0^{-1}$. Recall that $\mathcal{P}(K^{-1},A)$ is a collection of all dyadic squares with side length $K^{-1}$ in a set $A$ and we sometimes use $\mathcal{P}(K^{-1})$ for $\mathcal{P}(K^{-1},[-1,1]^2)$.
Recall also that Lemma \ref{linecount} states that 
for every square with side length $K_{d+1}^{-1}$ there are at most $M$ bad strips with width $2 \cdot 10^{-10d}\epsilon_0 K_L^{-1}$ intersecting the square.

For every $(d+2)$-tuple $\alpha=(\alpha_0,\ldots,\alpha_{d+1}) \in (0,1)^{d+2}$, we {say that $x\in \R^3$ is} an $\alpha$-broad point of $Ef$ if 
\begin{equation}\label{0511e3.2}
    \max_{L_1,\ldots,L_{M} \in \mathbb{L} } \Big|\sum_{\substack{ \tau \in \mathcal{P}(K^{-1},\, \cap_{i=1}^{M} L_i) } } Ef_{\tau}(x)\Big| \leq \alpha_{d+1} |Ef(x)|
\end{equation}
and
\begin{equation}\label{0507e3.3}
\begin{split}
    &\max_{\tau_j \in \mathcal{P}(K_j^{-1})} |Ef_{\tau_j}(x)|
    \\&+
    \max_{\upsilon_j \in \mathcal{P}(K_j^{-1})}
        \max_{L_1,\ldots,L_{M} \in \mathbb{L} } \Big|\sum_{\substack{ \tau \in \mathcal{P}(K^{-1},\,
        \upsilon_j \cap (\cap_{i=1}^{M} L_i ))}} Ef_{\tau}(x)\Big|
    \leq \alpha_j |Ef(x)|
\end{split}
\end{equation}
for every $j=0,\ldots,d+1$. Here, the bad strips $L_i$ do not need to be distinct. Also, the second term on the left hand side of \eqref{0507e3.3} is introduced for some technical reason.

For $\alpha \in (0,1)^{d+2}$ and $r \in \R$, we use the notation $r\alpha=(r\alpha_0,\ldots,r\alpha_{d+1})$.
We let $\br Ef(x)$ denote the function which is $|Ef(x)|$ if $x$ is an $\alpha$-broad point of $Ef$ and 0 otherwise. Theorem \ref{theorem1.2} follows from the following broad function estimate.

\begin{thm}\label{broadfunctionestimate}
 For every $\epsilon>0$ there exist dyadic numbers $K,K_1,\ldots,K_d,K_L$ with 
\begin{equation}
    K_L=K_{d+1}(\epsilon) \ll K_d(\epsilon) \ll \cdots \ll  K_1(\epsilon) \ll K_0(\epsilon)=K
\end{equation} 
and a small number $\delta_{\mathrm{trans}} \in (0,\epsilon)$ so that the following holds true: For every $\alpha=(\alpha_0,\ldots,\alpha_{d+1})$ with  $1 \geq \alpha_j \geq K_j^{-\epsilon}$ and  $j=0,\ldots,d+1$, $R \geq 1$, ball $B_R\subset \R^3$ of radius $R$, and function $f:[-1,1]^2 \rightarrow \C$, it holds that 
\beq\label{broadfunctionestimate'}
\begin{split}
	\|\mathrm{Br}_{\alpha}Ef\|_{L^{3.25}(B_R)}
	\leq C_{\epsilon,d}R^{\delta_{\mathrm{trans}}(\sum_{j=0}^{d+1}
	\log{(K_j^{\epsilon}\alpha_j) })}R^{\epsilon}
	\|f\|_{L^2}^{\frac{12}{13}+\epsilon}
	\max_{\theta \in \mathcal{P}(R^{-1/2})}\|f\|_{L^2_{\mathrm{avg}}(\theta)}^{\frac{1}{13}-\epsilon}.
	\end{split}
	\endeq
	Here, the averaged $L^2$ norm is defined by
\begin{equation}
\|f\|_{\avg}:=
\Big( \frac{1}{|\theta|} \int_{\theta}|f(x)|^2\,dx \Big)^{1/2}.
\end{equation}
Moreover,  $\lim_{\epsilon \rightarrow 0}K_{d+1}(\epsilon) \rightarrow +\infty$, and $\delta_{\text{trans}}$ can be taken to be $\epsilon^6$.
\end{thm}

Let us assume Theorem \ref{broadfunctionestimate} and finish the proof of Theorem \ref{theorem1.2}. \begin{proof}[Proof of Theorem \ref{theorem1.2}] For simplicity, we use the notation $p_0=3.25$. Fix $3/13<\lambda<5/13$. The other case {where} $5/13 \leq \lambda \leq 1$ follows by H\"{o}lder's inequality. Assume that $\epsilon>0$ is sufficiently small so that 
\begin{equation}
    0<10^d\epsilon<\lambda-\frac{3}{13}.
\end{equation}
Our proof relies on an induction argument on {the} radius $R$. We assume that \eqref{eq_desired_restriction_estimate} holds true for all the  {radii} smaller than $R/2$ and aim to prove \eqref{eq_desired_restriction_estimate} for the radius $R$.
Take $\alpha_j=K_j^{-\epsilon}$. By the definition of {the} broad function, we obtain
\begin{equation}\label{0507e3.7}
\begin{split}
|Ef(x)| &\leq |\br{Ef}(x)|
+\sum_{j=0}^{d+1}
\alpha_j^{-1} \Bigg(
 \max_{\tau_j \in \mathcal{P} (K_j^{-1})}|Ef_{\tau_j}(x)|\\&
        +\max_{\upsilon_j \in \mathcal{P}(K_j^{-1})}
        \max_{L_1,\ldots,L_{M} \in \mathbb{L} } \Big|\sum_{\tau \in \mathcal{P}(K^{-1},\, \upsilon_j \cap(\cap_{i=1}^{M} L_i)) } Ef_{\tau}(x)\Big|\Bigg)
        \\&
        +
        (\alpha_{d+1})^{-1}
        \max_{L_1,\ldots,L_{M} \in \mathbb{L} } \Big|\sum_{\tau \in \mathcal{P}(K^{-1},\, \cap_{i=1}^{M} L_i) } Ef_{\tau}(x)\Big|.
\end{split}
\end{equation}
We raise both sides to the $p_0$-th power, integrate over $B_R$, replace the max by $l^{p_0}$-norms, and obtain
\begin{equation}
\begin{split}
    \int_{B_R}|Ef|^{p_0}
    & \leq C\int_{B_R}|\br{Ef}|^{p_0}
    +C\sum_{j=0}^{d+1}
\alpha_j^{-p_0} \Bigg(
 \sum_{\tau_j \in \mathcal{P} (K_j^{-1})}\int_{B_R}|Ef_{\tau_j}|^{p_0}
 \\&
        +\sum_{\upsilon_j \in \mathcal{P}(K_j^{-1})}
        \sum_{L_1,\ldots,L_{M} \in \mathbb{L} }\int_{B_R} \bigg|\sum_{\tau \in \mathcal{P}(K^{-1},\, \upsilon_j \cap(\cap_{i=1}^{M} L_i)) } Ef_{\tau}\bigg|^{p_0}\Bigg)
        \\&
        +C
        (\alpha_{d+1})^{-p_0}
        \sum_{L_1,\ldots,L_{M} \in \mathbb{L} }\int_{B_R} \bigg|\sum_{\tau \in \mathcal{P}(K^{-1},\, \cap_{i=1}^{M} L_i) } Ef_{\tau}\bigg|^{p_0},
\end{split}
\end{equation}
{for some universal constant $C$}. We apply Theorem \ref{broadfunctionestimate} to the first term and obtain 
\begin{equation}\label{3.10}
\begin{split}
\int_{B_R}|\mathrm{Br}_{\alpha}Ef|^{p_0}
&\leq CC_{\epsilon,d}^{p_0}R^{p_0\epsilon}
\|f\|_{L^2}^{\frac{12}{13}p_0+\epsilon p_0}\max_{\theta \in \mathcal{P}(R^{-1/2})}\|f\|_{\avg}^{\frac{p_0}{13}-\epsilon p_0}
\\&
\leq 10^{-d p_0}C_{\epsilon,d,\lambda}^{p_0}R^{p_0\epsilon}
\|f\|_{L^2}^{p_0(1-\lambda)}\|f\|_{\infty}^{p_0\lambda},
\end{split}
\end{equation}
provided that $C_{\epsilon,d, \lambda}$ in Theorem \ref{theorem1.2} is chosen such that 
\begin{equation}
    10^{2dp_0}CC_{\epsilon,d} \leq C_{\epsilon,d,\lambda}.
\end{equation}
The second inequality of \eqref{3.10} follows from {H\"older's inequality and from} the fact that $12/13>1-\lambda$. This takes care of the contribution from the first term. 

Let us bound the second term. Note that $p_0(1-\lambda)>2$. We apply Lemma \ref{squarerescaling} and the embedding $l^{p_0(1-\lambda)} \hookrightarrow l^2$ to obtain
\begin{equation}
\begin{split}
\alpha_j^{-p_0}
\sum_{\tau_j \in \mathcal{P} (K_j^{-1})} \int_{B_R}|Ef_{\tau_j}|^{p_0} &\leq
(K_j)^{-\epsilon}
C_{\epsilon,d,\lambda}^{p_0}R^{p_0\epsilon}
\sum_{\tau_j \in \mathcal{P} (K_j^{-1})}
\|f_{\tau_j}\|_{L^2}^{p_0(1-\lambda)}\|f_{\tau_j}\|_{L^{\infty}}^{p_0\lambda}
\\&
\leq
(K_j)^{-\epsilon}
C_{\epsilon,d,\lambda}^{p_0}R^{p_0\epsilon}
\|f\|_{L^2}^{p_0(1-\lambda)}\|f\|_{L^{\infty}}^{p_0\lambda}.
\end{split}
\end{equation}
Hence, the second term is also harmless as $K_j$ is sufficiently large. The third term can be dealt with by following the same argument. We leave out the details.

Let us bound the fourth term. By Lemma \ref{striprescaling} and {the} embedding $l^{p_0(1-\lambda)} \hookrightarrow l^2$, we obtain
\begin{equation}
\begin{split}
&
\alpha_{d+1}^{-p_0}
\sum_{L_1,\ldots,L_{M} \in \mathbb{L} }
\int_{B_R}\bigg| \sum_{\tau \in \mathcal{P}(K^{-1},\, \cap_{i=1}^{M} L_i) }Ef_{\tau}
\bigg|^{p_0} 
\\&
\leq
(K_{d+1})^{-\epsilon}
C_{\epsilon,d,\lambda}^{p_0}R^{p_0\epsilon}
\sum_{L_1,\ldots,L_{M} \in \mathbb{L} }
\|f_{\cap_{i=1}^{M} L_{i}}\|_{L^2}^{p_0(1-\lambda)}\|f\|_{L^{\infty}}^{p_0\lambda}
\\&
\leq
(K_{d+1})^{-\epsilon}
C_{\epsilon,d,\lambda}^{p_0}R^{p_0\epsilon}
\Big(\sum_{L_1,\ldots,L_{M}\in \mathbb{L} }
\|f_{\cap_{i=1}^{M} L_i }\|_{L^2}^2\Big) ^{\frac{p_0(1-\lambda)}{2}}\|f\|_{L^{\infty}}^{p_0\lambda}.
\end{split}
\end{equation}
By Lemma \ref{linecount}, it is further bounded by
\begin{equation}
R^{p_0\epsilon }
\epsilon_0^{-p_0}10^{20d p_0}(K_{d+1})^{-\epsilon}C_{\epsilon,d,\lambda}^{p_0}\|f\|_{L^2}^{p_0(1-\lambda)}\|f\|_{L^{\infty}}^{p_0\lambda}.
\end{equation}
{By taking $K_{d+1}$ to be sufficiently large, we finish the analysis of the last term.} 
This completes the proof of Theorem \ref{theorem1.2}.
\end{proof}

\subsection{Wave packet decomposition}
We briefly review the standard wave packet decomposition. Details can be found in Tao \cite{MR2033842} and Guth \cite{MR3454378}. We decompose the square $[-1,1]^2$ into smaller {dyadic} squares $\theta$ of side length $R^{-1/2}$. Let $w_{\theta}$ denote the left bottom corner of $\theta$. Let $v_{\theta}$ denote the normal vector to our surface $S$ at the point $(w_{\theta},h(w_\theta))$. Let $\T(\theta)$ denote a set of tubes covering $B_R$, that are parallel to $v_{\theta}$ with radius $R^{1/2+\delta}$ and length $CR$. Denote $\T:=\cup_{\theta \in \mathcal{P}(R^{-1/2}) }\T(\theta)$ and $\T(\tau):=\cup_{\theta \in \mathcal{P}(R^{-1/2},\tau)}\T(\theta)$. For each $T \in \T(\theta)$, let $v(T)$ denote the direction $v_{\theta}$ of the tube.

\begin{prop}[Wave packet decomposition]\label{wavepacketdecomposition}
If $f \in L^2([-1,1]^2)$ then for each $T \in \mathbb{T}$ we can choose a function $f_T$ so that the following holds true:
\begin{enumerate}
    \item If $T \in \mathbb{T}(\theta)$ then $\mathrm{supp}(f_{T}) \subset 3\theta$;
    \item  If $x \in B_R \setminus T$, then $|Ef_T(x)| \leq R^{-1000}\|f\|_2$;
    \item  For any $x \in B_R$, $|Ef(x)-\sum_{T \in \mathbb{T}}Ef_T(x)| \leq R^{-1000}\|f\|_{L^2}$;
    \item If $T_1,T_2 \in \mathbb{T}(\theta)$ and $T_1,T_2$ are disjoint, then $|\int f_{T_1}\bar{f}_{T_2}| \leq R^{-1000} \int_{\theta}|f|^2$;
    \item $\sum_{T \in \mathbb{T}(\theta)} \int_{[-1,1]^2}|f_T|^2 \lesssim \int_{\theta}|f|^2$.
\end{enumerate}
\end{prop}

For the proof of the wave packet decomposition, we refer to that of Proposition 2.6 in \cite{MR3454378}. 
\\

We decompose $[-1,1]^2$ into smaller squares $\tau$ with side length $K^{-1}$. Write $f=\sum_{\tau \in \mathcal{P}(K^{-1}) } f_{\tau}$ and apply the above wave packet decomposition to the function $f_{\tau}$. To simplify notation, $(f_{\tau})_{T}$ will be abbreviated to $f_{\tau,T}$.

\begin{lem}\label{changesumfitof}
Suppose that ${\T}_k \subset \T$ indexed by $k \in A$ for some index set $A$. If each tube $T$ belongs to at most $\mu$ of the subsets $\{\T_k\}_{k \in A} $ then for every $\theta$,
\begin{equation}
\sum_{k \in A}\int_{3\theta}\Big|\sum_{T \in \T_k}f_{\tau,T} \Big|^2 \lesssim \mu \int_{10\theta}|f_{\tau}|^2.
\end{equation}
\end{lem}
The proof of Lemma \ref{changesumfitof} is identical to that of Lemma 2.7 in \cite{MR3454378}. Hence, we leave out the details.
We record a special case of Lemma \ref{changesumfitof}.
\begin{lem}\label{getridofoi}
For every cap $\theta$, $\tau$, and subcollection $\T_k \subset \T$,
\begin{equation}
\int_{3\theta}\bigg| \sum_{T \in \T_k} f_{\tau,T}\bigg|^2 \lesssim \int_{10\theta}|f_{\tau}|^2.
\end{equation}
\end{lem}

\section{Proof of Theorem \ref{broadfunctionestimate}}

In this section we will prove Theorem \ref{broadfunctionestimate}. We may assume that  \eqref{broadfunctionestimate'} holds true for $\alpha_{d+1} \geq 10^{-100d}$ by taking $K_{d+1}$ large enough and for $\alpha_{j} \geq K_{j+1}^{-100d}$ for $j=0,\ldots,d$ by taking $K_j$ large enough. The constants $K_{d+1},\cdots,K_0$ will satisfy
\begin{equation}
1 \ll K_{d+1} \ll K_d \ll \cdots \ll K_1 \ll K_0.
\end{equation}
We will introduce {other} parameters, called $\delta$, $\dt$, $\dd$, such that
\begin{equation}
\dt \ll \dd \ll \delta \ll \epsilon.
\end{equation}
More explicitly, we take 
\begin{equation}
\dt=\epsilon^6,\  \dd=\epsilon^4,\ \text{ and } \delta=\epsilon^2.
\end{equation}
Moreover, recall from the last section that $M=10^{20d}\epsilon_0^{-1}$.

The proof is via an induction argument on the radius $R$. 
We assume that \eqref{broadfunctionestimate'} holds true for {all radii} $\leq R/2$ and aim to prove it for all radii $\leq R$.\\

We will utilize the polynomial partitioning lemma in \cite{MR3454378}.
\begin{thm}[Corollary 1.7 in \cite{MR3454378}] \label{polynomialpartitioning}
 Let $W$ be a non-negative $L^1$ function on $\R^n$. Then for any $D \in \Z^{+}$, there is a non-zero polynomial $P$ of degree at most $D$ so that $\R^n \setminus \mc{Z}(P)$ is a disjoint union of $\simeq D^n$ open sets $O_i$, and the integrals $\int_{O_i} W$ agree up to a factor of 2. Moreover, the polynomial $P$ is a product of non-singular polynomials.
\end{thm}

We apply Theorem \ref{polynomialpartitioning} to the function $1_{B_R}(\br(Ef))^{p_0}$ with $D=R^{\dd}$. Then there exists a polynomial $P$ such that
\begin{equation}
\R^3 \setminus \mc{Z}(P) = \bigsqcup_{i=1}^{\simeq D^3} O_i
\end{equation}
and
\begin{equation}
    \int_{O_{i'} \cap B_R }(\br{Ef})^{p_0}
    \simeq
    \int_{O_i \cap B_R }(\br{Ef})^{p_0}
\end{equation}
for every $i$ and $i'$. Here, $P$ is a product of non-singular polynomials. 

We define a wall $W$ to be the $R^{1/2+\delta}$-neighborhood of the variety $\mathcal{Z}(P)$. Define cells $O_i':=(O_i \cap B_R)\setminus W$. Denote the collection of all the tubes intersecting $O_i'$ by
\begin{equation}
\T_i := \{T \in \T: T \cap O_i' \neq \emptyset \}. 
\end{equation}
For every $\tau \in \mathcal{P}(K^{-1})$ and $\tau_j \in \mathcal{P}(K_j^{-1})$ with $j=1,\ldots,d+1$,
define 
\begin{equation}
    f_{\tau,i}:=\sum_{T \in \T_i}f_{\tau,T},\;\;\;
    f_{\tau_j,i}:=\sum_{\tau \in \mathcal{P}(K^{-1},\tau_j) }f_{\tau,i}, \;\;\;
    f_i:=\sum_{\tau \in \mathcal{P}(K^{-1}) } f_{\tau,i}.
\end{equation}

{To prove \eqref{broadfunctionestimate'},}
we decompose our ball $B_R$ into the cells and  the wall and obtain
\begin{equation}
\int_{B_R} (\mathrm{Br}_{\alpha}Ef)^{p_0} = 
\sum_{i=1}^{\simeq D^3}
\int_{{B_R \cap O_i'}} (\mathrm{Br}_{\alpha}Ef)^{p_0}+
\int_{B_R \cap W} (\mathrm{Br}_{\alpha}Ef)^{p_0}.
\end{equation}
{There are two scenarios.}
{We say that we are in \textit{the cellular case} if}
\begin{equation}\label{0522e4.9}
\sum_{i=1}^{\simeq D^3}
\int_{{B_R \cap O_i'}} (\mathrm{Br}_{\alpha}Ef)^{p_0} \geq
\int_{B_R \cap W} (\mathrm{Br}_{\alpha}Ef)^{p_0}.
\end{equation}
Otherwise, we say that we are in  \textit{the wall case}.

\subsection{Cellular case}
In this case,
there exists a subcollection $\mathcal{T}$ of the index set of the cells with cardinality $\simeq D^3$ such that for all $i \in \mathcal{T}$
\beq\label{cellstep1}
\int_{{B_R}} (\mathrm{Br}_{\alpha}Ef)^{p_0} \simeq D^3
\int_{{B_R \cap O_i'}} (\mathrm{Br}_{\alpha}Ef)^{p_0}.
\endeq
We refer to Lemma 4.1 in \cite{MR3653943} for more details.

\begin{lem}\label{changeftofi}
For every $i\in \mc{T}$ and every $x \in O_i'$, it holds that 
\begin{equation}
\mathrm{Br}_{\alpha}Ef(x)
\leq \mathrm{Br}_{2\alpha}Ef_i(x) + O(R^{-900}\|f\|_{L^2}).
\end{equation}
\end{lem}

\begin{proof}[Proof of Lemma \ref{changeftofi}]
Suppose that $x \in O_i'$ is an $\alpha$-broad point of $Ef.$
By item $(3)$ of Proposition \ref{wavepacketdecomposition}, we obtain
\begin{equation}
Ef_{\tau}(x)=\sum_{T \in \T}Ef_{\tau,T}(x)+O(R^{-1000}\|f_{\tau}\|_2).
\end{equation}
If $x \in T$, then $T$ intersects $O_i'$ and $T \in \T_i$. If $x \notin T$, then by  item $(2)$ of Proposition \ref{wavepacketdecomposition}, we have $|Ef_{\tau,T}(x)| \leq R^{-1000}\|f_{\tau}\|_2$. Hence, we have
\beq\label{pointwisechange}
Ef_{\tau}(x)=Ef_{\tau,i}(x)+O(R^{-990}\|f_{\tau}\|_2).
\endeq
By summing over $\tau$, we obtain
\beq\label{pointwisechange2}
Ef(x)=Ef_i(x)+O(R^{-990}\|f\|_2).
\endeq

It remains to show that $x$ is a $2\alpha$-broad point of $Ef_i$. We may assume that $|Ef(x)| \geq R^{-900}\|f\|_2$ and by the above inequality, we have $|Ef_i(x)| \geq R^{-890}\|f\|_2$. We need to show that 

\begin{equation}\label{4.4} 
    \max_{L_1,\ldots,L_{M} \in \mathbb{L} } \Big|\sum_{\tau \in \mathcal{P}(K^{-1},\, \cap_{i=1}^{M} L_i) } Ef_{\tau,i}(x)\Big| \leq 2\alpha_{d+1} |Ef_i(x)|
\end{equation}
and
\begin{equation}\label{4.5}
\begin{split}
    &\max_{\tau_j \in \mathcal{P} (K_j^{-1})} |Ef_{\tau_j,i}(x)|
    \\&+
    \max_{\upsilon_j \in \mathcal{P}(K_j^{-1})}
        \max_{L_1,\ldots,L_{M} \in \mathbb{L} } \Big|\sum_{\tau \in \mathcal{P}(K^{-1},\, \upsilon_j \cap(\cap_{i=1}^{M} L_i))} Ef_{\tau,i}(x)\Big|
    \leq 2\alpha_j |Ef_i(x)|
\end{split}
\end{equation}
for every $j=0,\ldots,d+1$.

Let us first prove \eqref{4.4}.
By \eqref{pointwisechange}, the left hand side of \eqref{4.4} is bounded by
\begin{equation}
    \max_{L_1,\ldots,L_{M} \in \mathbb{L} } \Big|\sum_{\tau \in \mathcal{P}(K^{-1},\, \cap_{i=1}^{M} L_i) } Ef_{\tau}(x)\Big|+O({R^{-980}}\|f\|_2)
\end{equation}
    which is further bounded by
\begin{equation}
\frac{11}{10}\alpha_{d+1}|Ef(x)|
\end{equation}
as $x$ is an $\alpha$-broad point of $Ef$.
By \eqref{pointwisechange2}, the above display is further bounded by
\begin{equation}
\frac{11}{10}\alpha_{d+1} |Ef_i(x)|+O({R^{-980}}\|f\|_2).
\end{equation}
Since the error term is bounded by $\frac{9}{10}\alpha_{d+1}|Ef_i(x)|$, this gives \eqref{4.4}. The second inequality \eqref{4.5} can be proved in the exactly same way. We leave out the details.
\end{proof}

We apply Lemma \ref{changeftofi} to each cell $O_i'$ and obtain
\begin{equation}
\int_{B_R \cap O_i' }
\Big(\mathrm{Br}_{\alpha}Ef(x)\Big)^{p_0}
\lesssim
\int_{B_R} \Big(\mathrm{Br}_{2\alpha}Ef_i(x) \Big)^{p_0}+O(R^{-900p_0}\|f\|_{L^2}^{p_0}).
\end{equation}
If the error dominates the first term on the right hand side, we automatically obtain Theorem \ref{broadfunctionestimate}. Hence, we may assume that the first term dominates the error. We decompose the ball $B_R$ into smaller balls of radius $R/2$ and apply the induction hypothesis on $R$ and obtain
\beq\label{cellstep3}
\begin{split}
\int_{B_R \cap O_i' }
\Big(\mathrm{Br}_{\alpha}Ef(x)\Big)^{p_0}
\lesssim
R^{\dt(\sum_j\log{(K_j^{\epsilon}2\alpha_j) } )p_0 } R^{\epsilon p_0}
\|f_i\|_{L^2}^{3+\epsilon p_0} \max_{\theta \in \mathcal{P}(R^{-1/2})}\|f_{i}\|_{\avg}^{\frac14-\epsilon p_0}.
\end{split}
\endeq
We apply Lemma \ref{getridofoi} to $\|f_{i}\|_{L^2_{\mathrm{avg}}(\theta)}$ and see that the above term is bounded by
\beq\label{cellstep4}
\begin{split}
&R^{\dt(\sum_j\log{(K_j^{\epsilon}2\alpha_j) })p_0} R^{\epsilon p_0}\|f_i\|_{L^2}^{3+\epsilon p_0} \max_{\theta \in \mathcal{P}(R^{-1/2})}\|f\|_{\avg}^{\frac14-\epsilon p_0}
\\&\lesssim R^{\dt(\sum_j\log{(K_j^{\epsilon}2\alpha_j )})p_0 }R^{\epsilon p_0}\big( \sum_{\tau \in \mathcal{P}(K^{-1})} \|f_{\tau,i}\|_{L^2}^2 \big)^{\frac{3+\epsilon p_0}{2}} \max_{\theta \in \mathcal{P}(R^{-1/2})}\|f\|_{\avg}^{\frac14-\epsilon p_0}.
\end{split}
\endeq
As a line can intersect $\mathcal{Z}(P)$ at most $(D+1)$-times (recall that the degree of {the} polynomial $P$ is $D$), we see that each tube $T \in \T$ intersects at most $(D+1)$ many cells $O_i'$. By this observation and Lemma \ref{changesumfitof}, we have 
\begin{equation}
    \sum_{i \in \mathcal{T} } \|f_{\tau,i}\|_{L^2}^2 \lesssim D \|f_{\tau}\|_{L^2}^2.
\end{equation}
Since the cardinality of $\mathcal{T}$ is $\simeq D^3$, we {can find} some $i_0 \in \mathcal{T}$ such that
\beq\label{cellstep5}
\sum_{\tau \in \mathcal{P}(K^{-1})} \|f_{\tau,i_0}\|_{L^2}^2 \lesssim D^{-2}\sum_{\tau \in \mathcal{P}(K^{-1})}\|f_{\tau}\|_{L^2}^2 \lesssim D^{-2}\|f\|_{L^2}^2. 
\endeq
By combining \eqref{cellstep1}, \eqref{cellstep3}, \eqref{cellstep4} and \eqref{cellstep5} with $i=i_0$, we obtain
\begin{equation} 
\begin{split}
    &\int_{B_R}(\mathrm{Br}_{\alpha}Ef)^{p_0}
    \\&
\lesssim
D^{-\epsilon p_0}R^{C\dt}R^{\dt(\sum_j\log{(K_j^{\epsilon}\alpha_j )})p_0 } R^{\epsilon p_0}
\|f\|_{L^2}^{{3+\epsilon p_0}} \max_{\theta \in \mathcal{P}(R^{-1/2})}\|f\|_{L_{\mathrm{avg}}^2(\theta)}^{\frac14-\epsilon p_0}.
\end{split}
\end{equation}
It suffices to note that
\begin{equation}
D^{-\epsilon}R^{C\dt} = R^{-\epsilon \dd +C\dt} \ll 1
\end{equation}
for sufficiently large $R$. This completes the proof of the cellular case.

\subsection{Wall case}
In the wall case, we need to show that
\begin{equation}\label{Wallcasegoal}
\begin{split}
	&\|\mathrm{Br}_{\alpha}Ef\|_{L^{3.25}(B_R \cap W)}
	\\&
	\leq 10^{-1} C_{\epsilon,d}R^{\dt(\sum_j \log{(K_j^{\epsilon}\alpha_j )} ) } R^{\epsilon}
	\|f\|_{L^2}^{\frac{12}{13}+\epsilon}
	\max_{\theta \in \mathcal{P}(R^{-1/2})}\|f\|_{\avg}^{\frac{1}{13}-\epsilon}.
	\end{split}
	\end{equation}
We will cover $B_R$ with $\simeq R^{3\delta}$ balls $B_k$ of radius $R^{1-\delta}$. Let $\T_{k,\mathrm{tang}}$ denote the collection of all tubes $T \in \T$ such that $B_k \cap W \cap T \neq \emptyset$ and
\begin{equation}
\mathrm{Angle}(v(T),\mathcal{T}_z(\mathcal{Z}(P))) \leq R^{-1/2+2\delta},
\end{equation}
for every $z$ that is a non-singular point of $\mathcal{Z}(P)$ lying in $2B_k \cap 10T$. Here $\mathcal{T}_z(\mathcal{Z}(P))$ is the tangent plane of $\mathcal{Z}(P)$ at the point $z$.
Let 
$\mathbb{T}_{k,\mathrm{trans}}$ denote the collection of all tubes $T \in \mathbb{T}$ such that 
$B_k \cap W \cap T \neq \emptyset$ and
\begin{equation}
\mathrm{Angle}(v(T),\mathcal{T}_z(\mathcal{Z}(P))) > R^{-1/2+2\delta},
\end{equation}
for some non-singular point $z$ of $\mathcal{Z}(P)$ lying in $2B_k \cap 10T$. We use the notations $\T_{k,-}$ and $\T_{k,+}$ for $\T_{k,\mathrm{tang}}$ and $\T_{k,\mathrm{trans}}$, respectively. It is easy to see that each tube $T$ with $B_k \cap W \cap T \neq \emptyset$ belongs to either $\T_{k,+}$ or $\T_{k,-}$. Denote $\T_{k,+}(\tau):=\T(\tau) \cap \T_{k,+}$ and $\T_{k,-}(\tau):=\T(\tau) \cap \T_{k,-}$.
We will use two geometric estimates. 

\begin{lem}[Lemma 3.5 and 3.6 in \cite{MR3454378}]\label{Translemma}
\begin{enumerate}
\item []
    \item Each tube $T \in \T$ belongs to at most $R^{O(\dd)}$ different sets $\T_{k,+}$.
    \item For each $k$, the number of different $\theta$ with $\T_{k,-} \cap \T(\theta) \neq \emptyset$ is at most $R^{1/2+O(\delta)}$.
\end{enumerate}
\end{lem}
For a subcollection $I$ of the squares $\tau$, we define $f_I$ by
\begin{equation}
f_I=\sum_{\tau \in I}f_{\tau},
\end{equation}
and define 
\begin{equation}
f_{\tau,k,+}:= \sum_{\substack{T \in \mathbb{T}_{k,+}(\tau) }}f_T,\;\;\; f_{k,+}:=\sum_{\tau \in \mathcal{P}(K^{-1})} f_{\tau,k,+},\;\;\;
f_{I,k,+}:=\sum_{\tau \in I } f_{\tau,k,+}.
\end{equation}
Moreover, define $f_{\tau,k,-}, f_{k,-}, f_{I,k,-}$ similarly.
We also define the bilinear function
\begin{equation}
\mathrm{Bil}(Ef):=\sum_{\substack{(\tau,\tau')\in \mathcal{P}(K^{-1}) \times \mathcal{P}(K^{-1}):  \\
(\tau,\tau') \, \text{is a good pair}}} |Ef_{\tau}|^{1/2}|Ef_{\tau'}|^{1/2}.
\end{equation}
Now we are ready to run the broad-narrow analysis of Bourgain and Guth \cite{MR2860188} and Guth \cite{MR3454378}. We will use geometric lemma \ref{badpairlemma}  here and see the motivation of our definition of broad points.

\begin{lem}\label{bourgainguth}
Suppose that $\alpha \in (0,1)^{d+2}$ satisfies 
\begin{equation}
    K_j^{-\epsilon} \leq \alpha_j \leq K_{j+1}^{-100d},\;\;\;\; K_{d+1}^{-\epsilon} \leq \alpha_{d+1} \leq 10^{-100d}
\end{equation}
for every $j=0,\ldots,d$.
Then for every $x \in B_k \cap W$
\begin{equation}
\begin{split}
\mathrm{Br}_{\alpha}Ef(x) &\leq 2\sum_I \mathrm{Br}_{200d^2\alpha}Ef_{I,k,+}(x)
\\&+K_0^{100}\mathrm{Bil}(Ef_{k,-})(x) + R^{-50}\|f\|_2.
\end{split}
\end{equation}
The summation $\sum_I$ runs over all possible collections of squares with sidelength $K_0^{-1}$. 
\end{lem}

\begin{proof}[Proof of Lemma \ref{bourgainguth}]
Suppose that $x$ is an $\alpha$-broad point of $Ef$. We may assume that
\begin{equation}
|Ef(x)| \geq R^{-100}\|f\|_2.
\end{equation}
Otherwise, the inequality is trivial. We define the non-significant tangential part
\begin{equation}\label{0522e4.36}
\mathcal{C}=\big\{\tau \in \mathcal{P}(K^{-1}) : |Ef_{\tau,k,-}(x)| < K^{-100}|Ef(x)| \big\}.  
\end{equation}
We consider several cases.

\subsubsection*{Case 1} $\mathcal{C}^c$ is empty. 

In this case, by the triangle inequality, we obtain
\begin{equation}
\begin{split}
|Ef(x)| &\leq  \Big|\sum_{\tau \in \mathcal{P}(K^{-1})}Ef_{\tau,k,+}(x)\Big|+\sum_{\tau \in \mathcal{P}(K^{-1})}|Ef_{\tau,k,-}(x)|+ R^{-200}\|f\|_2
\\& \leq \Big|\sum_{\tau \in \mathcal{P}(K^{-1})}Ef_{\tau,k,+}(x)\Big|+ K^{-50}|Ef(x)| +R^{-200}\|f\|_2.
\end{split}
\end{equation}
By rearranging the terms, we obtain 
\begin{equation}\label{4.10}
|Ef(x)| \leq 2\Big|\sum_{\tau \in \mathcal{P}(K^{-1})}Ef_{\tau,k,+}(x)\Big|.
\end{equation}
It remains to show that $x$ is a $200d^2\alpha$-broad point of $Ef_{k,+}$. By definition, it follows from
\begin{equation}\label{4.14}
    \max_{L_1,\ldots,L_{M} \in \mathbb{L} } \Big|\sum_{\tau \in \mathcal{P}(K^{-1},\, \cap_{i=1}^{M} L_i) } Ef_{\tau,k,+}(x)\Big| \leq 200d^2\alpha_{d+1} \Big|\sum_{\tau \in \mathcal{P}(K^{-1})}Ef_{\tau,k,+}(x)\Big|
\end{equation}
and
\begin{equation}\label{4.15}
\begin{split}
    &
    \max_{\upsilon_j \in \mathcal{P}(K_j^{-1})}
        \max_{L_1,\ldots,L_{M} \in \mathbb{L} } \Big|\sum_{\tau \in \mathcal{P}(K^{-1},\, \upsilon_j \cap(\cap_{i=1}^{M} L_i))} Ef_{\tau,k,+}(x)\Big|
    \\&+
    \max_{\tau_j \in \mathcal{P}(K_j^{-1})} |\sum_{\tau \in \mathcal{P}(K^{-1},\tau_j )} Ef_{\tau,k,+}(x)|
    \leq 200d^2\alpha_j \Big|\sum_{\tau \in \mathcal{P}(K^{-1})}Ef_{\tau,k,+}(x)\Big|
\end{split}
\end{equation}
for every $j=0,\ldots,d+1$.
Let us prove \eqref{4.14}. By \eqref{4.10} it suffices to show that
\begin{equation}
    \max_{L_1,\ldots,L_{M} \in \mathbb{L} } \Big|\sum_{\tau \in \mathcal{P}(K^{-1},\, \cap_{i=1}^{M} L_i) } Ef_{\tau,k,+}(x)\Big| \leq 100d^2\alpha_{d+1} |Ef(x)|.
\end{equation}
It follows by the fact that $\tau \in \mathcal{C}$ and $x$ is an $\alpha$-broad point of $Ef(x)$. This gives the proof of \eqref{4.14}. The inequality \eqref{4.15} can be proved in exactly the same way. We leave out the details.

\subsubsection*{Case 2}
$\mathcal{C}^c$ is not empty.

We pick $\tau^* \in \mathcal{C}^c$ and take the dyadic square containing $\tau^*$ with sidelength $K_{d+1}^{-1}$. We take a square $\upsilon$ with sidelength $8K_{d+1}^{-1}$ so that it contains the dyadic square and the distance between the boundary of the dyadic square and the boundary of $\upsilon$ becomes greater than $2K_{d+1}^{-1}$. We may assume that the square $\upsilon$ is a union of 64 dyadic squares with side length $K_{d+1}^{-1}$. We have two subcases. 

\subsubsection*{Case 2.1}
$
\bigcup_{\tau \in \mathcal{C}^c}\tau \subset  \upsilon.
$

Note that all the elements $\tau \subset \upsilon^c$ belong to $\mathcal{C}$.
Since $x$ is an $\alpha$-broad point of $Ef$ and $x \in B_k$, we obtain
\begin{equation}
\begin{split}
|Ef(x)| &\leq \Big|\sum_{\tau \in \mathcal{P}(K^{-1},\, \upsilon) } Ef_{\tau}(x)\Big| + \Big|\sum_{\tau \in \mathcal{P}(K^{-1},\, \upsilon^c)} Ef_{\tau}(x)\Big|
\\&
\leq 64 \alpha_{d+1} |Ef(x)|+
\Big|\sum_{\tau\in \mathcal{P}(K^{-1},\, \upsilon^c )} Ef_{\tau}(x)\Big|
\\&
\leq
\big(65
\alpha_{d+1}+K^{-50}\big) |Ef(x)|+
\Big|\sum_{\tau \in \mathcal{P}(K^{-1},\, \upsilon^c) }Ef_{\tau,k,+}(x)\Big|+R^{-200}\|f\|_2.
\end{split}
\end{equation}
Recall that $\alpha_{d+1} \leq 10^{-100d}$. This implies
\beq\label{middlestep1}
|Ef(x)| \leq 2\Big|\sum_{\tau\in \mathcal{P}(K^{-1},\, \upsilon^c) }Ef_{\tau,k,+}(x)\Big|.
\endeq
It remains to show that $x$ is a $200d^2\alpha$-broad point of the function  $\sum_{\tau \subset \upsilon^c }Ef_{\tau,k,+}$. By definition and by applying \eqref{middlestep1} and the fact that $\tau\in \mc{C}$, it suffices to show
\begin{equation}\label{4.16}
    \max_{L_1,\ldots,L_{M} \in \mathbb{L} } \Big|\sum_{\substack{ \tau \in \mathcal{P}(K^{-1},\, (\cap_{i=1}^{M} L_i) \cap  \upsilon^c) } } Ef_{\tau}(x)\Big| \leq 90d^2\alpha_{d+1} |Ef(x)|
\end{equation}
and
\begin{equation}\label{4.17}
\begin{split}
    &\max_{\tau_j \in \mathcal{P} (K_j^{-1})} |Ef_{\tau_j}(x)|
    \\&+
    \max_{\upsilon_j \in \mathcal{P}(K_j^{-1})}
        \max_{L_1,\ldots,L_{M} \in \mathbb{L} } \Big|\sum_{\substack{\tau \in \mathcal{P}(K^{-1},\,
        \upsilon_j \cap (\cap_{i=1}^{M} L_i )) }} Ef_{\tau}(x)\Big|
    \leq 90d^2\alpha_j |Ef(x)|,
\end{split}
\end{equation}
for every $j=0,\ldots,d+1$. Note that \eqref{4.17} immediately follows from the assumption that $x$ is an $\alpha$-broad point of $Ef(x)$.
Let us prove \eqref{4.16}. By the triangle inequality, the left hand side of \eqref{4.16} is bounded by
\begin{equation}
\begin{split}
    \max_{L_1,\ldots,L_{M} \in \mathbb{L} } \Big|\sum_{\tau \in \mathcal{P}(K^{-1},\, \cap_{i=1}^{M} L_i)} Ef_{\tau}(x)\Big|
    +
    \max_{L_1,\ldots,L_{M} \in \mathbb{L} } \Big|\sum_{\tau \in \mathcal{P}(K^{-1},\, \upsilon \cap(\cap_{i=1}^{M} L_i))  } Ef_{\tau}(x)\Big|.
\end{split}
\end{equation}
Since $\upsilon$ is a union of 64 dyadic squares with side length $K_{d+1}^{-1}$ and $x$ is an $\alpha$-broad point of $Ef$, by the triangle inequality, it is further bounded by $65\alpha_{d+1}|Ef(x)|$ and this gives the proof of \eqref{4.16}.

\subsubsection*{Case 2.2}
There exists $\tau^{**} \in \mathcal{C}^c$ such that $\mathrm{dist}(\tau^*,\tau^{**}) \geq 2K_{d+1}^{-1}$.  

If
 $\mathcal{C}^{c}$ contains two squares $\tau_1$ and $\tau_2$ such that the pair $(\tau_1,\tau_2)$ is a good pair, then we obtain
\begin{equation}
|Ef(x)| \leq K^{100}|Ef_{\tau_1,k,-}(x)|^{1/2}|Ef_{\tau_2,k,-}(x)|^{1/2} \leq K^{100}\mathrm{Bil}(Ef_{k,-})(x).
\end{equation}
This gives the desired result.
Hence, we may assume that
there does not exist two squares $\tau_1,\tau_2 \in \mathcal{C}^c$ such that $(\tau_1,\tau_2)$ is  a good pair.

Since $x \in W$ and  non-singular points are dense on $\mathcal{Z}(P)$, by the definition of tangent wave packets, we can find a non-singular point $z$ of $\mathcal{Z}(P)$ such that 
\begin{equation}
|x-z| \leq 2R^{1/2+\delta} \text{ and }
\mathrm{Angle}(v(T),\mathcal{T}_z(\mathcal{Z}(P))) \leq R^{-1/2+2\delta}
\end{equation} 
for every $T \in \T_{k,-}$ with $z \in 10T$.
 Since $\tau^*,\tau^{**} \in \mathcal{C}^c$, they contain points such that the angles between the normal vectors of the surface associated with $h$ at the points and $\mathcal{T}_z(\mathcal{Z}(P))$ are smaller than $2R^{-1/2+2\delta}$. We denote the points in $\tau^{*}$ and $\tau^{**}$ by $\zeta^*$ and $\zeta^{**}$, respectively. Write $\overline{\zeta^{**}}:=(\overline{\zeta^{**}_1},\overline{\zeta^{**}_2}):=T_{\zeta^*}(\zeta^{**})$. Here, $T_{\zeta^*}$ is the map defined in \eqref{0422e2.12}. Without loss of generality, we may assume that $|\overline{\zeta^{**}_1}| \geq |\overline{\zeta^{**}_2}|$. 
 Note that
 \begin{equation}\label{4.30}
\begin{split}
    \{\theta \in \mathcal{P}(R^{-1/2})&: \mathrm{Angle}(T_z(\mathcal{Z}(P)),v_{\theta}) \leq 2R^{-1/2+2\delta} \}
    \\& \subset
    \{T_{\zeta^*}^{-1}(\xi,\eta): |P_{\zeta^*,\zeta^{**}}( \xi,\eta)|<R^{-1/2+O(\delta)}\}
\end{split}
 \end{equation}
because \eqref{badpairvariety2} says that $|P_{\zeta^*,\zeta^{**}}(\xi,\eta)|$ is comparable to the volume of the parallelogram generated by  the normal vectors of the surface associated with $h$ at $\zeta^*,\zeta^{**}$, and {$T_{\zeta^*}^{-1}(\xi,\eta)$}. Here, $v_{\theta}$ is the normal vector to the surface $S$ associated with $h$ at the point $(w_{\theta},h(w_{\theta}))$ and $w_{\theta}$ is the bottom corner of $\theta$. Recall that
\begin{equation}
    P_{\zeta^*,\zeta^{**}}(\xi,\eta)=\partial_1h_{\zeta^*}(\xi,\eta)\partial_2 h_{\zeta^*}(\overline{\zeta^{**}})-\partial_2h_{\zeta^*}(\xi,\eta)\partial_1 h_{\zeta^*}(\overline{\zeta^{**}}).
\end{equation}
Since $\partial_1h_{\zeta^*}(\xi,\eta)=\eta \cdot (1+O(\epsilon_0))+\partial_1h_{\zeta^*}(\xi,0)$ and $\partial_2h_{\zeta^*}(\xi,\eta)=\xi \cdot (1+O(\epsilon_0))+\partial_2h_{\zeta^*}(0,\eta)$, by \eqref{secondderivativelowerbound}, the set on the second line of \eqref{4.30} is contained in
\begin{equation}
    T_{\zeta^*}^{-1}(
    \mathcal{N}_{R^{-1/2+O(\delta)}}(\mathcal{Z}(P_{\zeta^*,\zeta^{**}}))).
\end{equation}
Hence, for every $\tau \in \mathcal{C}^c$ with $\mathrm{dist}(\tau^*,\tau) \geq 2K_{d+1}^{-1}$, there exists
\begin{equation}
    \zeta=\zeta(\tau) \in 2\tau \cap T_{\zeta^*}^{-1}(\mathcal{Z}(P_{\zeta^*,\zeta^{**}})).
\end{equation}
Since $T_{\zeta^*}(\zeta^*),T_{\zeta^*}(\zeta^{**}),T_{\zeta^*}(\zeta) \in \mathcal{Z}(P_{\zeta^*,\zeta^{**}})$,  we obtain
\begin{equation}\label{transversality}
    \det
    \begin{pmatrix}
    \partial_1h_{\zeta^*}(T_{\zeta^*}(\zeta^*)) &  \partial_2h_{\zeta^*}(T_{\zeta^*}(\zeta^*)) & -1
    \\
    \partial_1h_{\zeta^*}(T_{\zeta^*}(\zeta^{**})) &  \partial_2h_{\zeta^*}(T_{\zeta^*}(\zeta^{**})) & -1
    \\\partial_1h_{\zeta^*}(T_{\zeta^*}(\zeta)) &  \partial_2h_{\zeta^*}(T_{\zeta^*}(\zeta)) & -1
    \end{pmatrix}=0.
\end{equation}
Since $(\tau^*,\tau)$ is a bad pair, $(\zeta^*,\zeta)$ is a bad pair.
By \eqref{badpairvariety2} and \eqref{transversality}, we obtain $\mathcal{Z}(P_{\zeta^*,\zeta^{**}})=\mathcal{Z}(P_{\zeta^*,\zeta})$ and $\mathbb{B}(\zeta^*,\zeta)=\mathbb{B}(\zeta^*,\zeta^{**})$. Since $(\zeta^*,\zeta)$ is a bad pair, $T_{\zeta^*}(\zeta) \in \mathbb{B}({\zeta^*,\zeta})=\mathbb{B}({\zeta^*,\zeta^{**}})$.
We apply Lemma \ref{badpairlemma}. We consider the first case in Lemma \ref{badpairlemma}: There exist some squares $\upsilon_{j,i}' \in \mathcal{P}(K_j^{-1},\R^2)$  such that
\begin{equation}
\begin{split}
    &\{T_{\zeta^*}(\zeta)=T_{\zeta^*}(\zeta(\tau)): \tau \in \mathcal{C}^c,\; \mathrm{dist}(\tau^*,\tau) \geq 2K_{d+1}^{-1} \} 
    \\&\subset \mathbb{B}(\zeta^*,\zeta^{**})
    \subset 
\Big(
    \bigcup_{i=1}^{d^2}T_{\zeta^*}(\upsilon_{d+1,i}')\Big) \cup
    \Big(\bigcup_{j=2}^{d-1} \bigcup_{i=1}^{O(K_{j+1}^{3}) }T_{\zeta^{*}}(\upsilon_{j,i}')\Big).
\end{split}
\end{equation}
We take $T_{\zeta^*}^{-1}$ on the both sides, enlarge squares $\upsilon_{j,i}'$, include the square $\upsilon$ defined at the discussion at the beginning of Case 2, and we obtain
\begin{equation}
\bigcup_{\tau \in \mathcal{C}^c} \tau \subset  
\Big(
    \bigcup_{i=1}^{27d^2}\upsilon_{d+1
    ,i}'\Big) \cup
    \Big(\bigcup_{j=2}^{d-1} \bigcup_{i=1}^{K_{j+1}^{4} }\upsilon_{j,i}'\Big)=:\mathrm{Bad}
\end{equation}
for some $\upsilon_{j,i}' \in \mathcal{P}(K_j^{-1},[-1,1]^2)$.

By the triangle inequality,
\begin{equation}\label{4.36}
|Ef(x)| \leq  \Big|\sum_{\tau \in \mathcal{P}(K^{-1},\, \mathrm{Bad}) }Ef_{\tau}(x)\Big|+
\Big|\sum_{\tau \in \mathcal{P}(K^{-1}):\tau \cap \mathrm{Bad}^c \neq \emptyset }Ef_{\tau}(x)\Big|.
\end{equation}
By  the definition of an $\alpha$-broad point, the upper bound of $\alpha_j$, and the triangle inequality, we obtain
\begin{equation}\label{0501.423}
\begin{split}
    &\Big|\sum_{\tau \in \mathcal{P}(K^{-1},\, \mathrm{Bad}) }Ef_{\tau}(x)\Big|
    \\&\leq
    \sum_{i=1}^{27d^2}
    \Big|\sum_{\tau \in \mathcal{P}(K^{-1},\, \mathrm{\upsilon_{d+1,i}'}) }Ef_{\tau}(x)\Big|+
    \sum_{j=2}^{d-1}
    \sum_{i=1}^{K_{j+1}^4}
    \Big|\sum_{\tau \in \mathcal{P}(K^{-1},\, \upsilon_{j,i}' )}Ef_{\tau}(x)\Big|
    \\&\leq 
    \big(27d^2 \alpha_{d+1}+\sum_{j=2}^{d-1}K_{j+1}^4\alpha_{j} \big)|Ef(x)| \leq 2^{-1}|Ef(x)|.
\end{split}
\end{equation}
Hence, by combining \eqref{4.36} and \eqref{0501.423}, and by $x \in B_k$, we obtain
\begin{equation}\label{4.28}
|Ef(x)| \leq 2
\Big|\sum_{\tau \in \mathcal{P}(K^{-1}):\tau \cap \mathrm{Bad}^c \neq \emptyset }Ef_{\tau,k,+}(x)\Big|.
\end{equation}
It remains to show that $x$ is a $200d^2\alpha$-broad point of
$\Big|\sum_{\tau \cap \mathrm{Bad}^c \neq \emptyset }Ef_{\tau,k,+}(x)\Big|$. 
By the above inequality \eqref{4.28} and by the fact that all the squares $\tau$ with $\tau \cap \mathrm{Bad}^c \neq \emptyset$ belong to $\mathcal{C}$, it suffices to show that
\begin{equation}\label{4.21}
    \max_{L_1,\ldots,L_{M} \in \mathbb{L} } \Big|\sum_{\substack{\tau \in \mathcal{P}(K^{-1},\, \cap_{i=1}^{M} L_i):\\ \tau \cap \mathrm{Bad}^c \neq \emptyset } } Ef_{\tau}(x)\Big| \leq 90d^2\alpha_{d+1} |Ef(x)|
\end{equation}
and
\begin{equation}\label{4.22}
\begin{split}
    &\max_{\tau_j \in \mathcal{P} (K_j^{-1})} |Ef_{\tau_j}(x)|
    \\&+
    \max_{\upsilon_j \in \mathcal{P}(K_j^{-1})}
        \max_{L_1,\ldots,L_{M} \in \mathbb{L} } \Big|\sum_{\substack{ \tau \in \mathcal{P}(K^{-1},\,
        \upsilon_j \cap (\cap_{i=1}^{M} L_i )):\\\tau \cap \mathrm{Bad}^c \neq \emptyset} } Ef_{\tau}(x)\Big|
    \leq 90d^2\alpha_j |Ef(x)|
\end{split}
\end{equation}
for every $j=0,\ldots,d+1$. 

Let us prove \eqref{4.21} first. By the triangle inequality, the left hand side of \eqref{4.21} is bounded by
\begin{equation}
    \begin{split}
        \max_{L_1,\ldots,L_{M} \in \mathbb{L} } \Big|\sum_{\tau \in \mathcal{P}(K^{-1},\, \cap_{i=1}^{M} L_i) } Ef_{\tau}(x)\Big|
        +\max_{L_1,\ldots,L_{M} \in \mathbb{L} }
        \Big|\sum_{\tau \in \mathcal{P}(K^{-1},\, (\cap_{i=1}^{M}L_i) \cap \mathrm{Bad})}Ef_{\tau} \Big|.
    \end{split}
\end{equation}
Since $x$ is an $\alpha$-broad point of $Ef$, as in \eqref{0501.423}, it is further bounded by
\begin{equation}
    \Big(30d^2\alpha_{d+1}+\sum_{j=2}^{d-1}K_{j+1}^4\alpha_j \Big)|Ef(x)|.
\end{equation}
Since this is bounded by $100d^2\alpha_{d+1}|Ef(x)|$ by the lower bound and upper bound of $\alpha_j$,
this gives the proof of \eqref{4.21}. 

Let us prove \eqref{4.22}.
The first term on the left hand side of \eqref{4.22} can be bounded by $\alpha_j|Ef(x)|$ by the definition of an $\alpha$-broad point. By the triangle inequality, the second term is bounded by
\begin{equation}\label{4.25}
    \begin{split}
        &\max_{\upsilon_j \in \mathcal{P}(K_j^{-1})}
        \max_{L_1,\ldots,L_{M} \in \mathbb{L} } \Big|\sum_{\tau \in \mathcal{P}(K^{-1},\, \upsilon_j \cap(\cap_{i=1}^{M} L_i)) } Ef_{\tau}(x)\Big|
        \\&+
        \max_{\upsilon_j \in \mathcal{P}(K_j^{-1})}
        \max_{L_1,\ldots,L_{M} \in \mathbb{L} } \Big|\sum_{\substack{ \tau \in \mathcal{P}(K^{-1},\,
        \upsilon_j \cap (\cap_{i=1}^{M} L_i ) \cap \mathrm{Bad}})} Ef_{\tau}(x)\Big|.
    \end{split}
\end{equation}
The first term of \eqref{4.25} can be bounded by $\alpha_j|Ef(x)|$ by the definition of an $\alpha$-broad point. For $j=0,1,2$, $\mathrm{Bad}$ does not play a role and the second term of \eqref{4.25} can be dealt with by the definition of an $\alpha$-broad point and it is bounded by $\alpha_j|Ef(x)|$.
For $j=3,\ldots,d$, by the triangle inequality, the second term of \eqref{4.25} is bounded by
\begin{equation}
\begin{split}
    &\max_{\upsilon_j \in \mathcal{P}(K_j^{-1})}
        \max_{L_1,\ldots,L_{M} \in \mathbb{L} } \Big|\sum_{\tau \in \mathcal{P}(K^{-1},\, \upsilon_j \cap(\cap_{i=1}^{M} L_i)) } Ef_{\tau}(x)\Big|
        \\&+
    \sum_{l=2}^{j-1}
    \sum_{i=1}^{K_{l+1}^4}
        \max_{L_1,\ldots,L_{M} \in \mathbb{L} } \Big|\sum_{\substack{ \tau \in \mathcal{P}(K^{-1},
        \upsilon_{l,i} \cap (\cap_{i=1}^{M} L_i )  })} Ef_{\tau}(x)\Big|.
    \end{split}
\end{equation}
We can apply the definition of an $\alpha$-broad point and it is bounded by
\begin{equation}
   \Big(\alpha_j+ \sum_{l=2}^{j-1}K_{l+1}^4 \alpha_l \Big)|Ef(x)|.
\end{equation}
This is futher bounded by $200d^2\alpha_{j}|Ef(x)|$.
For $j=d+1$, by the triangle inequality, the second term of \eqref{4.25} is bounded by
\begin{equation}
    \Big(50d^2\alpha_{d+1}+ \sum_{l=2}^{d-1}K_{l+1}^4 \alpha_l \Big)|Ef(x)|.
\end{equation}
This is futher bounded by $200d^2\alpha_{d+1}|Ef(x)|$.
This completes the proof of \eqref{4.22} and finishes the discussion for the first case of Lemma \ref{badpairlemma}.
\medskip

Let us consider the second case in Lemma \ref{badpairlemma}: There exist bad strips $L_j$ such that
\begin{equation}
    \bigcup_{\tau \in \mathcal{C}^c}\tau \subset \upsilon  \cup
    L_1' \cup L_2'=:\mathrm{Bad}_1,
\end{equation}
where $\upsilon$ is the square with side length $8K_{d+1}^{-1}$ defined at the beginning of the discussion of Case 2.
By the triangle inequality,
\begin{equation}
|Ef(x)| \leq  \Big|\sum_{\tau \in \mathcal{P}(K^{-1}, \mathrm{Bad}_1) }Ef_{\tau}(x)\Big|+
\Big|\sum_{\tau \in \mathcal{P}(K^{-1}):\tau \cap \mathrm{Bad}_1^c \neq \emptyset }Ef_{\tau}(x)\Big|.
\end{equation}
By inclusion-exclusion principle,
\begin{equation}
\begin{split}
    &\Big|\sum_{\tau \in \mathcal{P}(K^{-1},\, \mathrm{Bad}_1 )}Ef_{\tau}(x)\Big|
    \\&\leq 
    \Big|\sum_{\tau \in \mathcal{P}(K^{-1},\, L_1') }Ef_{\tau}(x)\Big|+
    \Big|\sum_{\tau \in \mathcal{P}(K^{-1},\, L_2') }Ef_{\tau}(x)\Big|+
    \Big|\sum_{\tau \in \mathcal{P}(K^{-1},\, L_1' \cap L_2') }Ef_{\tau}(x)\Big|
    \\&
    +\Big|\sum_{\tau \in \mathcal{P}(K^{-1},\, \upsilon) }Ef_{\tau}(x)\Big|
    +\Big|\sum_{\tau \in \mathcal{P}(K^{-1},\, L_1' \cap \upsilon) }Ef_{\tau}(x)\Big|+
    \Big|\sum_{\tau \in \mathcal{P}(K^{-1},\, L_2' \cap \upsilon) }Ef_{\tau}(x)\Big|
    \\&
    +
    \Big|\sum_{\tau \in \mathcal{P}(K^{-1},\, L_1' \cap L_2' \cap \upsilon) }Ef_{\tau}(x)\Big|.
\end{split}
\end{equation}
By the definition of an $\alpha$-broad point, it is futher bounded by $300\alpha_{d+1}|Ef(x)|$. Hence, we obtain
\begin{equation}\label{4.49}
    |Ef(x)| \leq 2
    \Big|\sum_{\tau \in \mathcal{P}(K^{-1}): \tau \cap \mathrm{Bad}_1^c \neq \emptyset }Ef_{\tau}(x)\Big|.
\end{equation}
It remains to show that $x$ is a $200d^2\alpha$-broad point of $\Big|\sum_{\tau \cap \mathrm{Bad}_1^c \neq \emptyset }Ef_{\tau}(x)\Big|$. It can be proved as in \eqref{4.21} and \eqref{4.22}. By definition and \eqref{4.49}, it suffices to prove 
\begin{equation}\label{4.50}
    \max_{L_1,\ldots,L_{M} \in \mathbb{L} } \Big|\sum_{\substack{ \tau \in \mathcal{P}(K^{-1},\,\cap_{i=1}^{M} L_i):\\ \tau \cap \mathrm{Bad}_1^c \neq \emptyset } } Ef_{\tau}(x)\Big| \leq 90d^2\alpha_{d+1} |Ef(x)|
\end{equation}
and
\begin{equation}\label{4.51}
\begin{split}
    &\max_{\tau_j \in \mathcal{P} (K_j^{-1})} |Ef_{\tau_j}(x)|
    \\&+
    \max_{\upsilon_j \in \mathcal{P}(K_j^{-1})}
        \max_{L_1,\ldots,L_{M} \in \mathbb{L} } \Big|\sum_{\substack{\tau \in \mathcal{P}(K^{-1},\,
        \upsilon_j \cap (\cap_{i=1}^{M} L_i )):\\\tau \cap \mathrm{Bad}_1^c \neq \emptyset} } Ef_{\tau}(x)\Big|
    \leq 90d^2\alpha_j |Ef(x)|
\end{split}
\end{equation}
for every $j=0,\ldots,d+1$. 

Let us first prove \eqref{4.50}. By inclusion-exclusion principle, the left hand side of \eqref{4.50} is bounded by
\begin{equation}
\begin{split}
    &\max_{L_1,\ldots,L_{M} \in \mathbb{L} } \Big|\sum_{\tau \in \mathcal{P}(K^{-1},\, \cap_{i=1}^{M} L_i) } Ef_{\tau}(x)\Big|
    \\&+
    \max_{L_1,\ldots,L_{M} \in \mathbb{L} } \Big|\sum_{\tau \in \mathcal{P}(K^{-1},\, \upsilon_0 \cap(\cap_{i=1}^{M} L_i)) } Ef_{\tau}(x)\Big|
    \\&
    +\sum_{j=1}^{2} \max_{L_1,\ldots,L_{M} \in \mathbb{L} } \Big|\sum_{\tau \in \mathcal{P}(K^{-1},\, (\cap_{i=1}^{M} L_i) \cap L_j' )}Ef_{\tau}(x)\Big|
    \\&+\max_{L_1,\ldots,L_{M} \in \mathbb{L} } \Big|\sum_{\tau \in \mathcal{P}(K^{-1},\, (\cap_{i=1}^{M} L_i) \cap L_1' \cap L_2') } Ef_{\tau}(x)\Big|
    \\&
    +\sum_{j=1}^{2} \max_{L_1,\ldots,L_{M} \in \mathbb{L} } \Big|\sum_{ \tau \in \mathcal{P}(K^{-1},\, (\cap_{i=1}^{M} L_i) \cap L_j' \cap \upsilon) }Ef_{\tau}(x)\Big|
    \\&+\max_{L_1,\ldots,L_{M} \in \mathbb{L} } \Big|\sum_{ \tau \in \mathcal{P}(K^{-1},\, (\cap_{i=1}^{M} L_i) \cap L_1' \cap L_2'\cap \upsilon) } Ef_{\tau}(x)\Big|.
\end{split}
\end{equation}
Since $L_1'$ and $L_2'$ are bad strips and at each point $x$ there are at most $M$ bad strips intersecting $x$, the above term is further bounded by
\begin{equation}
\begin{split}
     &4\max_{L_1,\ldots,L_{M} \in \mathbb{L} } \Big|\sum_{\tau \in \mathcal{P}(K^{-1},\, \cap_{i=1}^{M} L_i) } Ef_{\tau}(x)\Big|
     \\&+4
     \max_{L_1,\ldots,L_{M} \in \mathbb{L} } \Big|\sum_{\tau \in \mathcal{P}(K^{-1},\, \upsilon \cap(\cap_{i=1}^{M} L_i)) } Ef_{\tau}(x)\Big|
     .
\end{split}
\end{equation}
Since $x$ is an $\alpha$-broad point of $Ef$, it is bounded by $300\alpha_{d+1}|Ef(x)|$ and finishes the proof of \eqref{4.50}. The proof of \eqref{4.51} is identical to that for $\eqref{4.50}$. We leave out the details.
\end{proof}

By raising to the $p_0$-th power on both sides of Lemma \ref{bourgainguth}, integrating over $B_k \cap W$ and summing over all the balls $B_k$, we obtain
\begin{equation}
\begin{split}
\int_{B_R \cap W}
|\mathrm{Br}_{\alpha}Ef(x)|^{p_0}
&\lesssim K^{500} \sum_{k,I}\int_{B_k \cap W} \Big( \mathrm{Br}_{200d^2\alpha }Ef_{I,k,+} \Big)^{p_0}
\\& +K^{500}\sum_k \int_{B_k \cap W}\mathrm{Bil}(Ef_{k,-})^{p_0} 
+ R^{-50}\|f\|_2^{p_0}.
\end{split}
\end{equation}
Hence, by the inequality
\begin{equation}
     \max_{\theta \in \mathcal{P}(R^{-1/2})}\|f\|_{\avg} \lesssim R^{1/2}\|f\|_2,
\end{equation}
the wall estimate \eqref{Wallcasegoal} follows from the estimate of the transverse wave packets
\beq\label{transversewall}
\begin{split}
&\sum_{k,I}\int_{B_k \cap W} \Big( \mathrm{Br}_{200d^2\alpha}Ef_{I,k,+}\Big)^{p_0} 
\\&\leq C(K) R^{-\delta \epsilon} R^{\dt (\sum_j\log{(K_j^{\epsilon}\alpha_j)})p_0 }
R^{\epsilon p_0} 
\|f\|_{L^2}^{(\frac{12}{13}+\epsilon)p_0}
	\max_{\theta \in \mathcal{P}(R^{-1/2})}\|f\|_{\avg}^{(\frac{1}{13}-\epsilon)p_0},
\end{split}
\endeq
and the estimate of the bilinear operator
\beq\label{tangentialwall}
\sum_k \int_{B_k \cap W}\mathrm{Bil}(Ef_{k,-})^{p_0}
\leq C(K)R^{C\delta}
\|f\|_{L^2}^{(\frac{12}{13})p_0}
	\max_{\theta \in \mathcal{P}(R^{-1/2})}\|f\|_{\avg}^{(\frac{1}{13})p_0}.
\endeq

\subsubsection{The estimates for the transverse wave packets}

Let us prove \eqref{transversewall}. 
Recall that $B_k$ is a ball of radius $R^{1-\delta}$.
By the induction on the radius $R$ and Lemma \ref{getridofoi}, we obtain
\begin{equation} 
\begin{split}
\int_{B_k \cap W}
|\mathrm{Br}_{200d^2\alpha}Ef_{I,k,+}(x)|^{p_0} 
&\leq
CC_{\epsilon}^{p_0}R^{C\dt}R^{\dt (\sum_j\log{(K_j^{\epsilon}\alpha_j))p_0}}R^{(1-\delta)\epsilon p_0}
\\& \times
\|f_{I,k,+}\|_{L^2}^{3+\epsilon p_0}
	\max_{\theta \in \mathcal{P}(R^{-1/2})}\|f\|_{\avg}^{0.25-\epsilon p_0}.
	\end{split}
\end{equation}
By summing over $B_k \subset B_R$ and by embedding $l^{3+ \epsilon p_0} \hookrightarrow l^2$, we obtain
\beq\label{12.1}
\begin{split}
&\sum_k\int_{B_k \cap W}
|\mathrm{Br}_{200d^2\alpha}Ef_{I,k,+}(x)|^{p_0} 
\\&\leq
C
C_{\epsilon}^{p_0}R^{C\dt}R^{p_0\dt (\sum_j\log{(K_j^{\epsilon}\alpha_j))}}R^{(1-\delta)\epsilon p_0}
\Big(\sum_{k,\tau} \|f_{\tau,k,+}\|_{L^2}^2 \Big)^{\frac{3+\epsilon p_0}{2}}
	\max_{\theta \in \mathcal{P}(R^{-1/2})}\|f\|_{\avg}^{\frac14-\epsilon p_0}.
	\end{split}
\endeq
We apply Lemma \ref{Translemma} with Lemma \ref{changesumfitof} and obtain
\beq\label{12.2}
\sum_k \int |f_{\tau,k,+}|^2 \lesssim R^{C\dd}\int |f_{\tau}|^2 +R^{-100}\|f_{\tau}\|_{2}^2.
\endeq
By combining \eqref{12.1} and \eqref{12.2}, we obtain
\begin{equation}
\begin{split}
&\sum_k\int_{B_k \cap W}
|\mathrm{Br}_{200d^2\alpha}Ef_{I,k,+}(x)|^{p_0} 
\\&\lesssim
C_{\epsilon}^{p_0}R^{C\dd}R^{C\dt}
R^{p_0\dt (\sum_j\log{(K_j^{\epsilon}\alpha_j))}}R^{(1-\delta)\epsilon p_0}
 \|f\|_{L^2}^{3+\epsilon p_0}
	\max_{\theta \in \mathcal{P}(R^{-1/2})}\|f\|_{\avg}^{\frac14-\epsilon p_0}.
	\end{split}
\end{equation}
Note that
\begin{equation}
R^{C\dd}
R^{C\dt}R^{-\delta\epsilon p_0} \ll R^{-\delta \epsilon}.
\end{equation}
This completes the proof of \eqref{transversewall}. 

\subsubsection{The estimates for the bilinear operator}

Let us prove \eqref{tangentialwall}.
By replacing the summation over $k$ by the supremum, it suffices to show that
\beq\label{Bilinearestimate}
\int_{B_k \cap W}\mathrm{Bil}(Ef_{k,-})^{p_0}
\leq C(K)R^{C\delta}
\|f\|_{L^2}^{(\frac{12}{13})p_0}
	\max_{\theta \in \mathcal{P}(R^{-1/2})}\|f\|_{\avg}^{(\frac{1}{13})p_0}
\endeq
for every $k$.
We cover $B_k \cap W$ with balls $Q$ of radius $R^{1/2}$. For each $Q$, let
\begin{equation}
\mathbb{T}_{k,-,Q}:=\{ T \in   \mathbb{T}_{k,-}: T \cap Q \neq \emptyset\}.
\end{equation}
On a ball $Q$, we obtain
\begin{equation}
Ef_{\tau,k,-}(x)= \sum_{T \in \mathbb{T}_{k,-,Q} }Ef_{\tau,T}(x) + O(R^{-100}\|f\|_2).
\end{equation}

Since a wall could be curved, the tangent wave packets are not necessarily coplanar. However, the tangent wave packets intersecting a ball $Q$ are always coplanar in the sense that they are contained in some $R^{1/2+\delta}$-neighborhood of a hyperplane by the following reason. Since $Q \cap W$ is non-empty, we can choose a point $z \in \mathcal{Z}(P)$ in the $R^{1/2+\delta}$-neighborhood of $Q$. Since the non-singular points are dense in $\mathcal{Z}(P)$ as $P$ is a product of non-singular polynomials, we may assume that $z$ is a non-singular point. By definition, for every $T \in \T_{k,-,Q}$ and $z \in 10T \cap 2B_k \cap \mathcal{Z}(P)$, the angle between $v(T)$ and $\mathcal{T}_z(\mathcal{Z}(P))$ is $O(R^{-1/2+O(\delta)})$. Hence, all the elements of $\T_{k,-,Q}$ are contained in the $R^{1/2+O(\delta)}$-neighborhood of the hyperplane $\mathcal{T}_z(\mathcal{Z}(P))$. 

Since the wave packets $T \in \T_{k,-,Q}$ are coplanar, the support of $f_{\tau,T}$ and $f_{\tau',T}$ must be clustered near a zero set of a polynomial. We have defined a good pair $(\tau,\tau')$ so that it enjoys the following property: Either the hypersurface on the zero set of the polynomial or the zero set of the polynomial itself is curved in such a way that we can perform $L^4$-argument either way. This idea will be clear in the proof of Proposition \ref{L4argument}.
\medskip

Let $\psi_Q$ be a smooth function such that $|\psi_Q(x)| \simeq 1$ on $x \in Q$ and $\mathcal{F}({\psi}_Q)$ is supported on $B_{cR^{-1/2}}(0)$ for some small number $c>0$. Here, $\mathcal{F}$ is the Fourier transform.
\begin{prop}
\label{L4argument}
Suppose that $(\tau,\tau')$ is a good pair. For any ball $Q$ of radius $R^{1/2}$ intersecting $B_k \cap W$,
\begin{equation}\label{0512e4.86}
\begin{split}
    \int_Q \Big|\sum_{T_1 \in \mathbb{T}_{k,-,Q} } &Ef_{\tau,T_1}\Big|^2 \Big|\sum_{T_2 \in \mathbb{T}_{k,-,Q} } Ef_{\tau',T_2}\Big|^2 
    \\&\lesssim C(K)
    R^{C\delta}
    \sum_{T_1 \in \mathbb{T}_{k,-,Q} }
    \sum_{T_2 \in \mathbb{T}_{k,-,Q} }
    \int_{\R^3}
    |\psi_Q|^4
    | Ef_{\tau,T_1}|^2 | Ef_{\tau',T_2}|^2.
\end{split}
\end{equation}
\end{prop}

\begin{proof}[Proof of Proposition \ref{L4argument}]
To apply the Fourier transform, we first replace the strict cut-off by a smooth cut-off and bound the left hand side of \eqref{0512e4.86} by   
\beq\label{prop4.7step1}
\begin{split}
        \int_{\R^3} |\psi_Q|^4 \Big| \sum_{T_1 \in \mathbb{T}_{k,-,Q} } Ef_{\tau,T_1}\Big|^2 \Big| \sum_{T_2 \in \mathbb{T}_{k,-,Q} } Ef_{\tau',T_2}\Big|^2.
\end{split}
\endeq
By the discussion above Proposition \ref{L4argument} and the discussion at the beginning of Case 2.2 of the proof of Lemma \ref{bourgainguth}, there exist points $\zeta \in \tau$ and $\zeta' \in \tau'$ such that  supports of $f_{\tau,T_1}$ and $f_{\tau',T_2}$ are contained in the set
\begin{equation}
 T_{\zeta}^{-1}(\mathcal{N}_{R^{-1/2+O(\delta)}}(\mathcal{Z}(P_{\zeta,\zeta'}))).   
\end{equation}
Let us put $\overline{\zeta'}:=(\overline{\zeta'_1},\overline{\zeta'_2}):=T_{\zeta}(\zeta')$. Without loss of generality, we may assume that $|\overline{\zeta'_1}| \geq |\overline{\zeta'_2}|$. By a change of variables we obtain
\begin{equation}
\begin{split}
|Ef_{\tau,T_1}(x)|&=
\Big|
\int_{[-CK^{-1},CK^{-1}]^{2}} f_{\tau,T_1}(T_{\zeta}^{-1}(\xi,\eta))e\big((Ax) \cdot (\xi,\eta,h_{\zeta}(\xi,\eta)\big)\,d\xi d\eta \Big|
\\&=|{E}_{S_2}g_{T_{\zeta}(\tau),T_1}(Ax)|
\end{split}
\end{equation}
for some linear transformation $A$ with determinant $\simeq 1$. Here $S_2$ is the surface corresponding to the polynomial $h_{\zeta}(\xi,\eta)$ and \begin{equation}
g_{T_{\zeta}(\tau),T_1}:=f_{\tau,T_1} \circ T_{\zeta}^{-1}.  
\end{equation}
We apply the same change of variables given by the matrix $T_{\zeta}^{-1}$ to $Ef_{\tau',T_2}$ and we define
\begin{equation}
    g_{T_{\zeta}(\tau'),T_1}:=f_{\tau',T_2} \circ T_{\zeta}^{-1}.
\end{equation}
Put $\phi_{Q}(x):=\psi_Q(A^{-1}x)$. After this change of variables, we obtain
\begin{equation}
\begin{split}
\eqref{prop4.7step1} \lesssim
 \int_{\R^3} |\phi_Q|^4 \Big| \sum_{T_1 \in \mathbb{T}_{k,-,Q} } {E}_{S_2}g_{T_{\zeta}(\tau),T_1}\Big|^2 \Big| \sum_{T_2 \in \mathbb{T}_{k,-,Q} } {E}_{S_2}g_{T_{\zeta}(\tau'),T_2}\Big|^2.
\end{split}
\end{equation}
We expand the $L^2$-norm and obtain
\begin{equation}
\begin{split} 
    \sum_{T_1,T_1' \in \mathbb{T}_{k,-,Q} }
    \sum_{T_2,T_2' \in \mathbb{T}_{k,-,Q} }
    \int_{\R^3}
    \big(
    &(\phi_Q
    {E}_{S_2}g_{T_{\zeta}(\tau),T_1})
    (\phi_Q{E}_{S_2}g_{T_{\zeta}(\tau'),T_2})\big)
    \\& \times
    \big(
    \overline{(\phi_Q {E}_{S_2}g_{T_{\zeta}(\tau),T_1'})(\phi_Q{E}_{S_2}g_{T_{\zeta}(\tau'),T_2'}})\big).
    \end{split}
\end{equation}
We apply Plancherel's theorem and the integration becomes
\begin{equation}
\begin{split}
    \int_{\R^3}
    \big(
    \mathcal{F}({\phi_Q
    {E}_{S_2}g_{T_{\zeta}(\tau),T_1}})
    &*
    \mathcal{F}(\phi_Q{E}_{S_2}g_{T_{\zeta}(\tau'),T_2})
     \big)
     \\& \times
    \big(
    \overline{\mathcal{F}(\phi_Q {E}_{S_2}g_{T_{\zeta}(\tau),T_1'})*\mathcal{F}(\phi_Q{E}_{S_2}g_{T_{\zeta}(\tau'),T_2'}})\big).
    \end{split} 
\end{equation}
We will perform an $L^4$ argument.
Denote by $w_1(T_i)$ and $w_1(T_i')$ the projections to the $\xi$-axis of the centers of the supports of $g_{T_{\zeta}(\tau),T_i}$ and $g_{T_{\zeta}(\tau'),T_i'}$.
The integration vanishes unless
\begin{align}\label{mid}
w_1(T_1)+w_1(T_2)&=w_1(T_1')+w_1(T_2') +O(R^{-1/2+O(\delta)}),
\\\label{4.97}
\eta_{\zeta,\zeta'}(w_1(T_1))+\eta_{\zeta,\zeta'}(w_1(T_2))&=
\eta_{\zeta,\zeta'}(w_1(T_1'))+\eta_{\zeta,\zeta'}(w_1(T_2'))+O(R^{-1/2+O(\delta)}),
\\
H_{\zeta,\zeta'}(w_1(T_1))+H_{\zeta,\zeta'}(w_1(T_2))&=H_{\zeta,\zeta'}(w_1(T_1'))+H_{\zeta,\zeta'}(w_1(T_2'))+O(R^{-1/2+O(\delta)}).
\end{align}
We claim that either
\begin{equation}\label{4.63}
    |\eta_{\zeta,\zeta'}'(\overline{\zeta_1'})-\eta_{\zeta,\zeta'}'(0)| > (2K_1)^{-1} 
\end{equation}
or
\begin{equation}\label{4.64}
    |H'_{\zeta,\zeta'}(\overline{\zeta'})|>(2K_1)^{-1}.
\end{equation}
Recall that $(\tau,\tau')$ is a good pair. Thus, there exist $\mu \in 2\tau$ and $\mu' \in 2\tau'$ such that $(\mu,\mu')$ is a good pair. Let us take $\overline{\mu'}:=(\overline{\mu_{1}'},\overline{\mu_{2}'}):=T_{\mu}(\mu')$. By the implicit function theorem, the set $\mathcal{Z}(P_{\mu,\mu'})$ can be written as a graph of some function ${\eta}_{\mu,\mu'}$.
Since $(\mu,\mu')$ is a good pair, the claim follows from
\begin{align}\label{4.65}
        &|\eta_{\zeta,\zeta'}'(\overline{\zeta_1'})-{\eta}_{\mu,\mu'}'(\overline{\mu_1'})| \lesssim K_{d+1}^2K^{-1},
        \\&\label{4.66}
        |\eta_{\zeta,\zeta'}'(0)-{\eta}_{\mu,\mu'}'(0)| \lesssim K_{d+1}K^{-1},
        \\&\label{4.67}
        |H'_{\zeta,\zeta'}(\overline{\zeta'})-H'_{\mu,\mu'}(\overline{\mu'})| \lesssim K_{d+1}K^{-1},
\end{align}
{provided that $K$ is sufficiently large, compared to $K_1$.} 
Before we enter the proof, let us mention that by continuity and by Taylor's expansion, we have 
\begin{equation}
\|T_{\zeta}-T_{\mu}\| \lesssim K^{-1} \text{ and } \|h_{\zeta}-h_{\mu}\| \lesssim K^{-1}.
\end{equation}

Let us first prove \eqref{4.65}. By \eqref{derivativeofeta}, 
$|\eta_{\zeta,\zeta'}'(\overline{\zeta_1'})-{\eta}_{\mu,\mu'}'(\overline{\mu_1'})|$ is equal to
\begin{equation}
    \Big|
\frac{\partial_{12}h_{\zeta}(\overline{\zeta'})\partial_1h_{\zeta}(\overline{\zeta'})-\partial_{11}h_{\zeta}(\overline{\zeta'})\partial_2h_{\zeta}(\overline{\zeta'}) }{\partial_{12}h_{\zeta}(\overline{\zeta'})\partial_2h_{\zeta}(\overline{\zeta'})-\partial_{22}h_{\zeta}(\overline{\zeta'})\partial_1h_{\zeta}(\overline{\zeta'})}
-
\frac{\partial_{12}h_{\mu}(\overline{\mu'})\partial_1h_{\mu}(\overline{\mu'})-\partial_{11}h_{\mu}(\overline{\mu'})\partial_2h_{\mu}(\overline{\mu'}) }{\partial_{12}h_{\mu}(\overline{\mu'})\partial_2h_{\mu}(\overline{\mu'})-\partial_{22}h_{\mu}(\overline{\mu'})\partial_1h_{\mu}(\overline{\mu'})}
\Big|.
\end{equation}
Let us write the above term as $\frac{A}{B}-\frac{C}{D}=\frac{A(D-B)+(A-C)B}{BD}$.
Since $|\partial_{12}h_{\zeta}(\overline{\zeta'})|\simeq |\partial_{12}h_{\mu}(\overline{\mu'})| \simeq 1$, by \eqref{secondderivativelowerbound}, we obtain $|BD|^{-1} \lesssim K_{d+1}^2$.
It suffices to prove that $|B-D| \lesssim K^{-1}$ and $|A-C| \lesssim K^{-1}$. This easily follows by the inequalities $\|h_{\zeta}-h_{\mu}\| \lesssim K^{-1}$ and $|\overline{\zeta'}-\overline{\mu'}|\leq 2K^{-1}$. This finishes the proof of \eqref{4.65}.

Let us next prove \eqref{4.66}. By the definition of $\eta_{\zeta,\zeta'}$ and ${\eta}_{\mu,\mu'}$,
\begin{equation}
\begin{split}
    |\eta_{\zeta,\zeta'}'(0)-{\eta}_{\mu,\mu'}'(0)|
    &=\Big|\frac{\partial_1h_{\zeta}(\overline{\zeta'}) }{\partial_2h_{\zeta}(\overline{\zeta'})}-
    \frac{\partial_1h_{\mu}(\overline{\mu'}) }{\partial_2h_{\mu}(\overline{\mu'})}
    \Big|
    \\&=
    \Big|\frac{\partial_1h_{\zeta}(\overline{\zeta'})\partial_2h_{\mu}(\overline{\mu'})-\partial_1h_{\mu}(\overline{\mu'})\partial_2h_{\zeta}(\overline{\zeta'}) }{\partial_2h_{\zeta}(\overline{\zeta'})\partial_2h_{\mu}(\overline{\mu'}) }
    \Big|.
\end{split}
\end{equation}
This is equal to
\begin{equation}
    \Big|\frac{\partial_1h_{\zeta}(\overline{\zeta'})\big( \partial_2h_{\mu}(\overline{\mu'})-\partial_2h_{\zeta}(\overline{\zeta'}) \big)+\big(\partial_1h_{\zeta}(\overline{\zeta'})  -\partial_1h_{\mu}(\overline{\mu'})\big)\partial_2h_{\zeta}(\overline{\zeta'}) }{\partial_2h_{\zeta}(\overline{\zeta'})\partial_2h_{\mu}(\overline{\mu'}) }
    \Big|.
\end{equation}
By \eqref{secondderivativelowerbound}, it is bounded by
\begin{equation}
    \lesssim K_{d+1}\Big(|\partial_2 h_{\mu}(\overline{\mu'})-\partial_2 h_{\zeta}(\overline{\zeta'})|+|\partial_1 h_{\zeta}(\overline{\zeta'})-\partial_1 h_{\mu}(\overline{\mu'})| \Big).
\end{equation}
Since $\|h_{\zeta}-h_{\mu}\| \lesssim K^{-1}$ and $|\overline{\zeta'}-\overline{\mu'}| \leq 2K^{-1}$, it is further bounded by $K_{d+1}K^{-1}$ and finishes the proof of \eqref{4.66}.

Lastly, let us prove \eqref{4.67}. By the chain rule,
\begin{equation}
\begin{split}
    &|H_{\zeta,\zeta'}'(\overline{\zeta'})-H_{\mu,\mu'}'(\overline{\mu'})|
    \\&=
    |\partial_1h_{\zeta}(\overline{\zeta'})+\partial_2h_{\zeta}(\overline{\zeta'})\eta_{\zeta,\zeta'}'(\overline{\zeta_1'})-\partial_1h_{\mu}(\overline{\mu'})-\partial_2h_{\mu}(\overline{\mu'}){\eta}_{\mu,\mu'}'(\overline{\mu_1'})|
    \\&
    \leq |\partial_1h_{\zeta}(\overline{\zeta'})-\partial_1h_{\mu}(\overline{\mu'})|
    +|\partial_2h_{\zeta}(\overline{\zeta'}){\eta}_{\zeta,\zeta'}'(\overline{\zeta_1'})-\partial_2h_{\mu}(\overline{\mu'}){\eta}_{\mu,\mu'}'(\overline{\mu_1'})|.
\end{split}
\end{equation}
It is further bounded by $K_{d+1}K^{-1}$ by the inequalities $\|h_{\zeta}-h_{\mu}\| \lesssim K^{-1}$, and $|\overline{\zeta'}-\overline{\mu'}| \leq 2K^{-1}$, and \eqref{4.65}. This finishes the proof of the claim.
\\

By the claim, either \eqref{4.63} or \eqref{4.64} holds true.
Since the side length of $\tau$ and $\tau'$ is $K^{-1}$, we obtain that
\begin{align}
    &\eta'_{\zeta,\zeta'}(\xi)=\eta'_{\zeta,\zeta'}(0)+O(K^{-1}) \text{  for every  } (\xi,\eta_{\zeta,\zeta'}(\xi)) \in T_{\zeta}(\tau),
    \\&
    \eta'_{{\zeta,\zeta'}}(\xi')=\eta'_{\zeta,\zeta'}(\overline{\zeta_1'})+O(K^{-1}) \text{  for every  } (\xi',\eta_{{\zeta,\zeta'}}(\xi')) \in T_{\zeta}(\tau').
\end{align}
Hence, we obtain either
\beq\label{l4case1}
|\eta'_{{\zeta,\zeta'}}(w_1(T_1))-\eta'_{\zeta,\zeta'}(w_1(T_2))| > \frac{1}{4K_1} \text{  for every  } T_i \in \T_{k,-,Q},
\endeq
or
\beq\label{l4case2}
|H_{\zeta,\zeta'}'(w_1(T_1))-H_{\zeta,\zeta'}'(w_1(T_2))| > \frac{1}{4K_1} \text{  for every  } T_i \in \T_{k,-,Q}.
\endeq
We consider only the case that \eqref{l4case1} holds true. The other case \eqref{l4case2} can be dealt with in exactly the same way.
By the relation \eqref{mid}, we may assume that $w_1(T_1)<w_1(T_1')<w_1(T_2')<w_1(T_2)$. We take two intervals $I=[w_1(T_1),w_1(T_1')]$ and $I'=[w_1(T_2'),w_1(T_2)]$. Note that $|I|=|I'|+O(R^{-1/2+O(\delta)})$. 
By \eqref{l4case1}, we obtain 
\begin{equation}
    |I|+|I'| \lesssim C(K)\Big| \int_I \eta_{\zeta,\zeta'}' - \int_{I'}\eta_{\zeta,\zeta'}'  \Big| +R^{-1/2+O(\delta)},
\end{equation}
and by the fundamental theorem of calculus, it is further bounded by
\begin{equation}
\begin{split}
C(K)
\Big|
\eta_{\zeta,\zeta'}(w_1(T_1'))-\eta_{\zeta,\zeta'}(w_1(T_1))-\eta_{\zeta,\zeta'}(w_1(T_2'))+\eta_{\zeta,\zeta'}(w_1(T_2))
\Big|+R^{-1/2+O(\delta)}.
\end{split}
\end{equation}
By the relation \eqref{4.97}, this is bounded by $C(K)R^{-1/2+O(\delta)}$ and we obtain $|w_1(T_i)-w_1(T_i')| \lesssim R^{-1/2+O(\delta)}$. Therefore,
\begin{equation}
\begin{split}
\eqref{prop4.7step1} 
 \lesssim
R^{C\delta}
 \sum_{T_1 \in \mathbb{T}_{k,-,Q} }
\sum_{T_2 \in \mathbb{T}_{k,-,Q} }
 \int_{\R^3} |\phi_Q|^4  \big| {E}_{S_2}g_{T_{\zeta}(\tau),T_1}\big|^2 \big|  {E}_{S_2}g_{T_{\zeta}(\tau'),T_2}\big|^2.
\end{split}
\end{equation}
We change back to the original variables and it gives the desired result.
\end{proof}

\begin{lem}\label{L4argument2}
Suppose that $(\tau,\tau')$ is a good pair.
For any ball $Q$ of radius $R^{1/2}$ intersecting $B_k \cap W$,
\begin{equation}
\begin{split}
    \int_{Q} |Ef_{\tau,k,-}|^2&|Ef_{\tau',k,-}|^2 \lesssim
    C(K)
    R^{O(\delta)}R^{-1/2}
    \\& \Big( \sum_{T_1 \in \mathbb{T}_{k,-,Q} }\|f_{\tau,T_1}\|_{2}^2 \Big)
    \Big( \sum_{T_2 \in \mathbb{T}_{k,-,Q}}\|f_{\tau',T_2}\|_{2}^2 \Big).
\end{split}
\end{equation}
\end{lem}

\begin{proof}
By Lemma \ref{L4argument}, we obtain
\begin{equation}
\begin{split}
    \int_Q \Big|\sum_{T_1 \in \mathbb{T}_{k,-,Q} } &Ef_{\tau,T_1}\Big|^2 \Big|\sum_{T_2 \in \mathbb{T}_{k,-,Q} } Ef_{\tau',T_2}\Big|^2 
    \\&\lesssim 
    R^{C\delta}
    \sum_{T_1 \in \mathbb{T}_{k,-,Q} }
    \sum_{T_2 \in \mathbb{T}_{k,-,Q} }
    \int_{\R^3}
    |\psi_Q|^4
    | Ef_{\tau,T_1}|^2 | Ef_{\tau',T_2}|^2.
\end{split}
\end{equation}
We take a smooth function $\psi_{B_R}$ such that 
$|\psi_{B_R}(x)| \simeq 1$ on $x \in B_R$ and $\mathcal{F}({\psi}_{B_R})$ is supported on $B_{cR^{-1}}(0)$ for some small number $c>0$. It suffices to show that
\begin{equation}
\int_{\R^3}|\psi_{B_R} Ef_{\tau,T_1}|^2 |\psi_{B_R} Ef_{\tau',T_2}|^2 \lesssim C(K)R^{-1/2}\|f_{\tau,T_1}\|_{2}^2\|f_{\tau',T_2}\|_{2}^2.
\end{equation}
By Plancherel's theorem,
it amounts to showing that
\begin{equation}
\int_{\R^3}
    |\mathcal{F}(\psi_{B_R} Ef_{\tau,T_1}) *
    \mathcal{F}(\psi_{B_R}Ef_{\tau',T_2})|^2 \lesssim
    C(K)
    R^{-1/2}\|f_{\tau,T_1}\|_2^2 \|f_{\tau',T_2}\|_2^2.
\end{equation}
Since the normal vectors of $T_1$ and $T_2$ are separated by $\gtrsim 1/K_L$ by the separation condition on $\tau_1$ and $\tau_2$, the measure of the intersection of translations of the support of $\mathcal{F}(\psi_{B_R} Ef_{\tau,T_1})$ and $\mathcal{F}(\psi_{B_R} Ef_{\tau',T_2})$ is $\lesssim K_L^{O(1)}R^{-5/2}$. Hence, we obtain two estimates
\begin{equation}
\begin{split}
    &\|\mathcal{F}(\psi_{B_R} Ef_{\tau,T_1})*\mathcal{F}(\psi_{B_R} Ef_{\tau',T_2})\|_{\infty} \lesssim K_L^{O(1)}R^{-1/2}\|f_{\tau,T_1} \|_{\infty}\|f_{\tau',T_2}\|_{\infty},
    \\&
    \|\mathcal{F}(\psi_{B_R} Ef_{\tau,T_1})* \mathcal{F}(\psi_{B_R} Ef_{\tau',T_2})\|_1 \lesssim \|f_{\tau,T_1} \|_{1}\|f_{\tau',T_2}\|_{1},
\end{split}
\end{equation}
and interpolating these two inequalities gives the desired inequality.
\end{proof}

We are ready to prove \eqref{Bilinearestimate}.
Let $(\tau_1,\tau_2)$ be a good pair.
We decompose $B_k\cap W$ into smaller balls $Q$ of raidus $R^{1/2}$ and obtain
\begin{equation}
\int_{B_k \cap W}|Ef_{\tau_1,k,-}|^2|Ef_{\tau_2,k,-}|^2=
\sum_{Q \subset B_k \cap W}
\int_{Q}|Ef_{\tau_1,k,-}|^2|Ef_{\tau_2,k,-}|^2.
\end{equation}
By applying Lemma \ref{L4argument2} to each ball $Q$, we obtain
\begin{equation}
\begin{split}
&\int_{B_k \cap W}|Ef_{\tau_1,k,-}|^2|Ef_{\tau_2,k,-}|^2
 \\&
 \lesssim
 C(K)
R^{C\delta}R^{-1/2}
\sum_{Q \subset B_k \cap W}
\Biggl(
\prod_{l=1}^{2}
\Big( \sum_{T_l \in \mathbb{T}_{k,-,Q} }\chi_{T_l}(Q)
\|f_{\tau_l,T_l}\|_{2}^2 \Big)\Biggr).
    \end{split}
\end{equation}
Here, $\chi_{T_l}(Q)=1$ if $T_l$ intersects $Q$ and $\chi_{T_l}(Q)=0$ otherwise. 
By interchanging the sums, it is further bounded by
\begin{equation}
\begin{split}
&
C(K)
R^{C\delta}R^{-1/2}
\sum_{T_1 \in \mathbb{T}_{k,-} }\sum_{T_2 \in \mathbb{T}_{k,-}}
\|f_{\tau_1,T_1}\|_2^2 \|f_{\tau_2,T_2}\|_2^2
\sum_{Q \subset B_k \cap W}
\chi_{T_1}(Q)\chi_{T_2}(Q)
\\&
\lesssim
C(K)
R^{2C\delta}R^{-1/2}
\sum_{T_1 \in \mathbb{T}_{k,-}}\sum_{T_2 \in \mathbb{T}_{k,-} }
\|f_{\tau_1,T_1}\|_2^2 \|f_{\tau_2,T_2}\|_2^2.
\end{split}
\end{equation}
By summing over all the good pairs, we obtain
\begin{equation}
\|\mathrm{Bil}(Ef_{k,-})\|_{L^4(B_k \cap W)} \lesssim
C(K)R^{2C\delta}R^{-1/8}\big(\sum_{\tau\in \mathcal{P}(K^{-1}) } \|f_{\tau,k,-}\|_2^2 \big)^{1/2}.
\end{equation}
By a trivial estimate, we obtain
\begin{equation}
\begin{split}
\|\mathrm{Bil}(Ef_{k,-})\|_{L^2(B_k \cap W)} &\lesssim 
C(K)  \Big(\sum_{\tau\in \mathcal{P}(K^{-1}) }\|Ef_{\tau,k,-}\|_{L^2(B_k)}^2 \Big)^{1/2}
\\& \lesssim C(K)R^{1/2}\big(\sum_{\tau\in \mathcal{P}(K^{-1}) } \|f_{\tau,k,-}\|_2^2 \big)^{1/2}.
\end{split}
\end{equation}
By interpolating these inequalities via H\"{o}lder's inequality, we obtain
\begin{equation}
\begin{split}
\|\mathrm{Bil}(Ef_{k,-})\|_{L^{p_0}(B_k \cap W)}
&\lesssim 
C(K)
\|\mathrm{Bil}(Ef_{k,-})\|_{L^{2}(B_k \cap W)}^{\frac{3}{13}}
\|\mathrm{Bil}(Ef_{k,-})\|_{L^{4}(B_k \cap W)}^{\frac{10}{13}}
\\&
\lesssim R^{C\delta}R^{\frac{1}{52}}\big(\sum_{\tau\in \mathcal{P}(K^{-1}) } \|f_{\tau,k,-}\|_2^2 \big)^{1/2}.
\end{split}
\end{equation}
By Lemma \ref{Translemma}, we obtain
\begin{equation}
\|f_{\tau,k,-}\|_2^2  \lesssim R^{-1/2+O(\delta)}\max_{\theta \in \mathcal{P}(R^{-1/2})} \|f_{\tau}\|_{L^2_{avg}(\theta)}^2.
\end{equation}
Hence, by combining this with Lemma \ref{getridofoi}, we obtain
\begin{equation}
\begin{split}
\|\mathrm{Bil}(Ef_{k,-})\|_{L^{p_0}(B_k \cap W)}
&
\lesssim C(K)R^{C\delta}
\big(\sum_{\tau \in \mathcal{P}(K^{-1}) } \|f_{\tau,k,-}\|_2^2 \big)^{\frac{6}{13}}
\max_{\theta \in \mathcal{P}(R^{-1/2})} \|f_{\tau}\|_{\avg}^{\frac{1}{13}}
\\&
\lesssim C(K)R^{C\delta}
 \|f\|_2^{\frac{12}{13}}
\max_{\theta \in \mathcal{P}(R^{-1/2})} \|f_{\tau}\|_{\avg}^{\frac{1}{13}}.
\end{split}
\end{equation}
This completes the proof of \eqref{Bilinearestimate}.

	\bibliography{reference}{}
	\bibliographystyle{alpha}

\medskip

\medskip

\noindent Department of Mathematics, University of Wisconsin-Madison and the Institute for Advanced Study\\
\emph{Email address}: 
shaomingguo@math.wisc.edu
\\

\noindent Department of Mathematics, University of Wisconsin-Madison\\
\emph{Email address}: coh28@wisc.edu

\end{document}